\documentclass{article}
\usepackage{dirtytalk}
\usepackage[margin=3cm]{geometry}
\usepackage[utf8]{inputenc}
\usepackage[english]{babel}
\usepackage{amsthm}
\usepackage{amsmath}
\usepackage{amssymb}
\usepackage{amsfonts}
\usepackage{xspace}
\usepackage{hyperref}
\usepackage{mathtools}
\usepackage{tikz-cd}
\usepackage{hyperref}
\usepackage{lipsum}
\usepackage[style=alphabetic]{biblatex}
\hypersetup{
    colorlinks=true,
    linkcolor=blue,
    filecolor=magenta,      
    urlcolor=cyan,
}
\usepackage[title]{appendix}
\theoremstyle{definition}
\theoremstyle{plain}
\newtheorem*{theorem*}{Theorem}
\newtheorem{theorem}{Theorem}[section]
\newtheorem{definition}[theorem]{Definition}
\newtheorem{lemma}[theorem]{Lemma}
\newtheorem{corollary}[theorem]{Corollary}
\theoremstyle{remark}
\newtheorem*{remark}{Remark}

\newcommand\blfootnote[1]{%
  \begingroup
  \renewcommand\thefootnote{}\footnote{#1}%
  \addtocounter{footnote}{-1}%
  \endgroup
}
\DeclareMathOperator{\Div}{Div}
\newcounter{Chapcounter}

\newcommand{\chapter}[1] 
{ {\centering          
  \addtocounter{Chapcounter}{1} \Large \underline{\textbf{ \color{blue} Chapter \theChapcounter: ~#1}} }   
  \addcontentsline{toc}{section}{ \color{blue} Chapter:~\theChapcounter~~ #1}    
}

\title{Convex Bodies associated to Linear series of Adelic Divisors on Quasi-Projective Varieties}
\author{Debam Biswas}
\date{}

\begin{document}

\maketitle
\begin{abstract}
In this article we define and study convex bodies associated to linear series of adelic divisors over quasi-projective varieties that have been introduced recently by Xinyi Yuan and Shou-Wu Zhang. Restricting our attention to big adelic divisors, we deduce properties of volumes obtained by Yuan and Zhang using different convex geometric arguments. We go on to define augmented base loci and  restricted volumes of adelic divisors following the works of Michael Nakamaye and develop a similar study using convex bodies to obtain analogous properties for restricted volumes. We closely follow methods developed originally by Robert Lazarsfeld and Mircea Mustață.
\end{abstract}
\section*{Keywords:}
\textbf{Adelic Divisors,\ Volumes and Restricted Volumes,\ Augmented Base locus,\ Okounkov Bodies }
\section*{Mathematics Subject Classification:}
\textbf{\ 14G40,\ 14C40,\ 52A27}
\blfootnote{\textbf{Debam Biswas,\ Department of Mathematics,\ University of Regensburg\\
Universitatstrasse 31, 93053, Regensburg}\\
Email:\ \href{mailto:debambiswas@gmail.com}{debambiswas@gmail.com}}
\section{Introduction}
The theory of Okounkov bodies to study linear systems of line bundles on a projective variety was introduced by Russian mathematician Andrei Okounkov in his articles \cite{okou1} and \cite{okou2}. Given a linear series of an ample line bundle on a projective variety, he introduced certain convex bodies, which later came to be known as \emph{Okounkov bodies} whose convex geometric properties encode interesting invariants of the graded series. In their article \cite{lazarsfeld2008convex} Robert Lazarsfeld and Mircea Mustață noticed that the constructions of Okounkov generalise from ample line bundles to arbitary big line bundles on projective varieties. In their paper \cite{lazarsfeld2008convex} they develop a systematic study of Okounkov bodies for big line bundles and prove various properties of volumes such as continuity, Fujita approximation and others. They also consider the notion of \emph{restricted volumes} along a closed sub-variety and prove properties analogous to those of ordinary volumes.\\
A crucial feature of the approach in \cite{lazarsfeld2008convex} is that the construction of Okounkov bodies makes sense even when the variety is not projective as long as we have a graded series of the space of global sections on our given line bundle. In this article we use this observation to construct Okounkov bodies for \say{compactified} line-bundles on quasi-projective varieties. In their recent pre-print \cite{yuan2021adelic} Xinyi Yuan and Shou-Wu Zhang introduced the notion of \emph{adelic divisors} on a quasi-projective variety $U$ over a field. They manage to put a topology on the space of all divisors which come from projective models $X_i$ of $U$ and consider all divisors which are ''compactified'' with this topology. In other words an \emph{adelic divisor} on a nornmal quasi projective variety $U$ is given by the data $\{X_i,D_i\}$  and a sequence of positive rationals $q_i$ converging to 0 where $X_i$ are projective models of $U$, $D_i$ are $\mathbb{Q}$-divisors on $X_i$ with $D_i|_U=D_j|_U$ for all $i,j$ such that the following \say{Cauchy condition} holds.
\[-q_jD_0\le D_i-D_j\le q_jD_0\ \forall\ i\ge j\]
Here inequalities signifiy effectivity relations holding in a common projective model (see section 2.4 of \cite{yuan2021adelic} for details). As a result of their consideration, given any divisor $D$ on $U$ and an adelic compactification on $D$ denoted by $\overline{D}$, we get a space of adelic global sections $H^0(U,\overline{D})$ which is a \textbf{finite dimensional} sub-space of all global sections $H^0(U,O(D))$. Hence we can consider the notions of volumes similarly to the projective case and it is shown in \cite{yuan2021adelic} that these volume functions shows properties analogous to the classical projective volumes (see \cite{yuan2021adelic}, section 5). However following the approach in \cite{lazarsfeld2008convex} in this article we construct \textbf{Okounkov bodies} $\Delta(\overline{D})$ for the graded series $\{H^0(U,m\overline{D})\}_{m\in\mathbb{N}}$. The construction is essentially a special case of the construction sketched in Definition 1.16 of \cite{lazarsfeld2008convex} where we take $W_m=H^0(U,m\overline{D})\subseteq H^0(U,O(D))$. If the divisor $\overline{D}$ is big \emph{i.e} it has positive volume as defined in \cite{yuan2021adelic}, we show that the Lebesgue volume of the body is essentially the same as the algebraic volume upto scaling. The first main theorem of our article is as follows
\begin{theorem*}[A]
Suppose we have a big adelic divisor $\overline{D}$ on a normal quasi-projective variety $U$ and suppose $\Delta(\overline{D})$ is the Okounkov body associated to $\overline{D}$. Furthermore suppose $\widehat{\emph{vol}}(\overline{D})$ be the adelic volume defined in Theorem 5.2.1 of \cite{yuan2021adelic}. Then we have 
$$\emph{vol}_{\mathbb{R}^d}(\Delta(\overline{D}))=\lim_{m\to\infty}\frac{\emph{dim}_K(\overline{H^0}(U,m\overline{D}))}{m^d}=\frac{1}{d!}\cdot \widehat{\emph{vol}}(\overline{D})$$
\end{theorem*}
Continuing with our analogy of the approach in the article \cite{lazarsfeld2008convex} we construct global bodies for adelic Okounkov bodies to study the variation of these bodies. Although we do not have finiteness of the  Néron–Severi space associated to adelic divisors, it turns out there exist a canonical global convex body whose fibers give Okounkov bodies even if this global body depends on some choices of divisors in contrast to in \cite{lazarsfeld2008convex}. This is the content of our next theorem
\begin{theorem*}[B]
Suppose $\overline{D}$ and $\overline{E}$ be adelic divisors on a normal quasi-projective variety $U$ such that $\overline{D}$ is big. Then there exists a convex body $\Delta(U)=\Delta(U,\overline{D},\overline{E})\subset \mathbb{R}^{d+2}$ with the property that for any $\vec{a}=(a_1,a_2)\in\mathbb{Q}^2$ with $a_1\overline{D}+a_2\overline{E}$ big, we have \[\Delta(a_1\overline{D}+a_2\overline{E})=\Delta(U)\cap(\mathbb{R}^d\times\{\vec{a}\})\]
\end{theorem*}
The above two theorems combine to prove that Theorem 5.2.1 of \cite{yuan2021adelic} for big adelic divisors using convex geometric methods and Okounkov bodies. This volume essentially measures the asymptotic growth of the global sections which arise as restrictions of global sections from the bigger variety just as in the classical projective case ( Furthermore they show that not only the volume of the big adelic divisor but also its Okounkov body constructed in this article is approximated (in terms of Hausdorff metric) by the corresponding Okounkov bodies of the projective models defining the divisor.\\
Next we go on to define the notions of restricted volumes of adelic divisors along a closed sub-variety $E$ of $U$ using Okounkov bodies. The restricted volume essentially measures the asymptotic growth of global sections of $\overline{D}|_E$ which arise as restrictions of sections of $\overline{D}$ over $U$ to $E$ analogously to the classical projective setting (see \cite{restrvoldef} for more details). Analogously to the projective case, we can form the convex geometric objects $\Gamma_{U|E}(\overline{D})$, $\Delta_{U|E}(\overline{D})$ and the algebraic objects $H^0(U|E,\overline{D})$, $\widehat{\text{vol}}_{U|E}(\overline{D})$ for a given adelic divisor $\overline{D}$. In order to have relations analogous to that of the adelic volume, we introduce the notion of {augmented base locus} of an adelic divisor in analogy with projective augmented base locus (see section 2.4, \cite{lazarsfeld2008convex}). Our definition, although being very similar to the projective case, requires some work to be shown well-defined. Since we do not have Serre finiteness on quasi-projective varieties, we have to use the main result of \cite{birkar} to show the well-definedness. We go on to show that when $E$ is not contained in the augmented base locus, the expected properties hold which is our next theorem
\begin{theorem*}[C]
Suppose $\overline{D}$ is an adelic divisor on a normal quasi-projective variety $U$ over $K$. Furthermore suppose $E$ is a closed irreducible sub-variety of $U$ not contained in the augmented base locus of $\overline{D}$. Then we have 
\[\emph{vol}_{\mathbb{R}^d}(\Delta_{U|E}(\overline{D}))=\lim_{m\to\infty}\frac{\emph{dim}_K(H^0(U|E,O_E(m\overline{D})))}{m^d}=\frac{1}{d!}\cdot \widehat{\emph{vol}}_{U|E}(\overline{D})\]
where $\emph{dim}(E)=d$
\end{theorem*}
We go on to show the existence of a global body even for restricted volumes and hence the variation of these restricted Okounkov bodies also has desirable satisfying properties like continuity etc.\\ \ \\
The organisation of the article is as follows. In the first three sections of the first chapter we review the notions of adelic divisors and their space of global sections following sub-section 2.4 in \cite{yuan2021adelic}. In the fourth section we construct the Okounkov bodies for adelic divisors and show some preleminary properties of them. In the fifth section we prove our first main theorem relating the algebraic adelic volumes with Euclidean volumes of their Okounkov bodies. In the next two sections we construct the global bodies and show that their fibers essentially gives the variation of Okounkov bodies in fixed directions. We also deduce certain corollaries of the existence of global bodies. In the first section of the second chapter we define augmented base locus of an adelic divisor. We go on to define the restricted volume of an adelic divisor along a closed sub-variety in the next section. We relate them to Euclidean volumes of restricted Okounkov bodies and show the existence of global bodies in analogy to the adelic volume in the next two sections. We end the chapter by obtaining certain corollaries of restricted volumes similar to those of ordinary volumes in chapter 1.
\subsection{Adelic divisors}
We begin by giving a short review of adelic divisors which are our main objects of interest in this article. We fix a quasi-projective varity $U$ over any field $k$. By a \emph{projective model} of $U$, we mean a projective variety $X$ over $k$ which contains $U$ as an open dense subset via an open immersion $U\xhookrightarrow{} X$. Given a projective model $X$ of $U$, we have the group of Cartier $\mathbb{Q}$-divisors denoted by $\text{Div}(X)_{\mathbb{Q}}=\text{Div}(X)\otimes_{\mathbb{Z}}\mathbb{Q}$. Then we consider the group of $(\mathbb{Q},\mathbb{Z})$-divisors $\text{Div}(X,U)$ as follows
\[\text{Div}(X,U)=\{(D,\mathcal{D})\in\text{Div}(U)\oplus \text{Div}(X)_{\mathbb{Q}}\mid \mathcal{D}|_U=D\ \text{in}\ \text{Div}(U)_{\mathbb{Q}}\}\]
where $\mathcal{D}|_U$ denotes the image of $\mathcal{D}$ under the pull-back morphism $\text{Div}(X)_{\mathbb{Q}}\rightarrow \text{Div}(U)_{\mathbb{Q}}$.\\ \ \\
Note that the set of \emph{all} projective models of a given $U$ for an inverse system which in turn makes the set of $(\mathbb{Q},\mathbb{Z})$-divisors into a directed system via pull-backs. Then we can form the direct limit to define the group of \emph{model divisors} as follows 
\[\text{Div}(U/k)_{\text{mod}}=\lim_{X}\text{Div}(X,U)\]
where above the direct limit is taken as $X$ varies over all projective models of $U$.
Next note that there is a notion of effectivity in both the groups $\text{Div}(X)_{\mathbb{Q}}$ and $\text{Div}(U)$ which induces a partial order on $\text{Div}(X,U)$ where $(D,\mathcal{D})\le (D',\mathcal{D}')$ if and only if both $D'-D$ and $\mathcal{D}'-\mathcal{D}$ are effective in $\text{Div}(U)$ and $\text{Div}(X)_{\mathbb{Q}}$ respectively. This partial order induces a partial order in $\text{Div}(U/k)_{\text{mod}}$ by passing to direct limits.\\
By a \emph{boundary divisor} of $U$ over $k$, we mean a tuple $(X_0,D_0)$ where $X_0$ is a projective model of $U$ and $D_0$ is an effective Cartier divisor on $X_0$ with $\text{Supp}(D_0)=X_0-U$. Note that such a boundary divisor always exists which can be seen by choosing any projective model $X_0'$ of $U$ and blowing-up $X_0'$ along the reduced center $X_0-U$. Then note for any non-zero rational $r\in\mathbb{Q}$ we can view $rD_0$ as an element of $\text{Div}(X_0,U)$ and hence as en element of $\text{Div}(U/k)$ by setting the component on $\text{Div}(U)$ to be 0.\\
We can finally put a norm, denoted by a \emph{boundary norm} on $\text{Div}(U/k)_{\text{mod}}$ as follows 
\[||\cdot||_{D_0}\colon \text{Div}(U/k)_{\text{mod}}\rightarrow [0,\infty]\]
\[||\overline{D}||_{D_0}=\text{inf}\{q\in \mathbb{Q}_{>0}\mid -qD_0\le \overline{D}\le qD_0\ \text{in}\ \text{Div}(U/k)_{\text{mod}}\}\]
It is shown Lemma 2.4.1 of \cite{yuan2021adelic} that $||\cdot||_{D_0}$ is actually a norm and the topology induced by it on $\text{Div}(U/k)_{\text{mod}}$ is independent of the chosen boundary divisor $(X_0,D_0)$. Hence we can talk about the \emph{boundary topology} on $\text{Div}(U/k)_{\text{mod}}$ as the topology induced by a boundary norm coming from \emph{any} boundary divisor. We finally define the \emph{adelic divisors}, denoted by $\widehat{\text{Div}}(U,K)$, as the completion of the topological space $\text{Div}(U/k)_{\text{mod}}$ with respect to the boundary topology described above. Note that then by an adelic divisor we mean the data $\{X_i,D_i\}$ where $X_i$ are projective models and $D_i\in \text{Div}(X_i,U)$ and a sequence of positive rationals $\{q_i\}$ converging to 0  satisfying the effectivity relations
\[-q_iD_0\le D_i-D_j\le q_iD_0\ \text{in}\ \text{Div}(U/k)_{\text{mod}}\ \text{for all}\ j\ge i\]
\begin{remark}
If we assume $U$ to be normal, we can choose the models $X_i$ to be normal and further embedding the group of Cartier divisors into Weil divisors we can look at $D_i$ just as elements of $\text{Div}(X_i)_{\mathbb{Q}}$ and the effectivity relation to be holding in just $\text{Div}(X_i)_{\mathbb{Q}}$ instead of $\text{Div}(U/k)_{\text{mod}}$. This is due to the fact that group of Weil divisors on $U$ has no torsion.
\end{remark}
\subsection{Space of global sections of an adelic divisor}
We fix an algebraically closed field $K$ and a normal quasi-projective variety $U$ over $K$. As described in the previous section, we have the notion of the group of \emph{adelic divisors}  $\widehat{\Div}(U,K)$ ( \cite{yuan2021adelic} sub-section 2.4.1 for more details) which are given by a compatible sequence of models $\{X_i,D_i\}$ such that $D_i$ are Cartier $\mathbb{Q}$-divisors on the projective models $X_i$ such that $D_i$ restrict to a Cartier divisor $D$ on $U$ and they satisfy the Cauchy condition with respect to a boundary divisor $D_0$ defined over a projective model $X_0$ \textit{i.e} there exists a sequence of positive rational numbers $\{q_j\}$ converging to 0 such that 
\begin{equation}
    \label{eq:cauchy}
    D_j-q_jD_0\le D_i\le D_j+q_jD_0\ \text{for all}\ i\ge j\,.
\end{equation}
where $D_0$ is an effective Cartier divisor on $X_0$ with support exactly equal to the complement of $U$ in $X_0$ and the above effectivity relations are considered in a common model.( for details see \cite{yuan2021adelic}, section 2) 
Note that the definition of adelic divisor does not depend on the particular choice of the boundary divisor $D_0$, as shown in \cite[Lemma~2.4.1]{yuan2021adelic}.
We denote this data by $\overline{D}$.
Given such an adelic divisor, we introduce the space of global sections
\[
H^0(U,\overline{D})=H^0(U,O(\overline{D}))=\{f\in\kappa(U)^\times\mid \text{div}(f)+\overline{D}\ge 0\}\cup\{0\}
\]
following \cite[section~5.1.2]{yuan2021adelic}.
In the above definition, $\text{div}(f)$ is the adelic divisor obtained by picking the divisor corresponding to $f\in\kappa(U)^{\times}=\kappa(X)^{\times}$ on any projective model $X$ of $U$, and $\text{div}(f)+\overline{D}\ge 0$ means that the left hand side can be represented by a sequence of effective divisors on the corresponding models.
\begin{remark}
It is shown in \cite[Lemma~5.1.7(2)]{yuan2021adelic} that this space is always finite dimensional. This will be our analogue for the usual space of global sections on which we construct Okounkov bodies. For this purpose, note that by restricting the effectivity relation $\text{div}(f)+\overline{D}\ge 0$ to $U$, we can identify $H^0(U,\overline{D})$ with a finite dimensional vector sub-space of the space of all sections $H^0(U,O(D))$( which in general is very large and infinite dimensional). This will always be our way of viewing the vector spaces $H^0(U,\overline{D})$.
\end{remark}

\subsection{Different notions of effective sections}
Note that $D_i$ can be viewed as a (model) adelic divisor $\overline{D_i}$ in $\widehat{\text{Div}}(U,K)$ and consequently we have the space of global sections $H^0(U,\overline{D_i})$ as before, where we put the overline to emphasize it is viewed as a model adelic divisor.
However viewing $D_i$ as a $\mathbb{Q}$-divisor on the projective variety $X_i$ we can also define the space of global sections as before
\[
H^0(X_i,D_i)^\prime=\{f\in\kappa(X_i)^\times\mid \text{div}(f)+D_i\ge 0\ \text{in}\ \text{Div}(X_i)_{\mathbb{Q}} \}\cup\{0\}\,.
\]
only by restricting our attention to the projective model $X_i$.
These two notions of effective sections can be different a-priori. However if we consider $U$ to be normal, then by \cite[Lemma 5.1.5 and Remark~5.1.6]{yuan2021adelic} they are canonically identified and we get that both these notions are the same.
Next we will obtain some inclusions.

\begin{lemma}
\label{lemma:incl}
We have the sequence of inclusions 
\[
H^0(X_j,k(D_j-q_jD_0))^\prime\hookrightarrow\ H^0(U,O(k\overline{D}))\hookrightarrow H^0(X_j,k(D_j+q_jD_0))^\prime
\]
for all $k\in \mathbb{N}$ and for all $j$.
\end{lemma}
\begin{proof}
Note that by our discussion above the two extremes of the sequence can be replaced by
\[
H^0(U,k(\overline{D_j}-q_j\overline{D_0}))\quad\text{and}\quad H^0(U,k(\overline{D_j}+q_j\overline{D_0}))
\]
respectively as $U$ is assumed to be normal.
Therefore, the statement is equivalent to the chain of inequalities
\[
\overline{D}_j-q_j\overline{D}_0\leq\overline{D}\leq\overline{D}_j+q_j\overline{D}_0\,,
\]
which is an immediate consequence of~\eqref{eq:cauchy}.
\end{proof}

Next we define the volume of an adelic line bundle following \cite[sub-section~5.2.2]{yuan2021adelic}.
\begin{definition}
\label{def:vol}
Given an adelic line bunde $\overline{D}$ on a quasi-projective variety $U$ as above, we define the \emph{volume of $\overline{D}$} as 
\[
\widehat{\emph{vol}}(\overline{D})=\limsup_{m\to\infty}\frac{\dim_K(\widehat{H}^0(U,m\overline{D}))}{m^d/d!}\,,
\]
where $d$ is the dimension of $U$. We call an adelic divisor \emph{big} if $\widehat{\emph{vol}}(\overline{D})>0$.
\end{definition}
We will primarily be interested in the Okounkov bodies of the big adelic divisors.
\begin{remark}
It is shown in \cite[Theorem~5.2.1(1)]{yuan2021adelic} that the $\limsup$ in Definition~\ref{def:vol} is actually a limit by using the fact that the volume is actually a limit of the volumes of the projective $\mathbb{Q}$-volumes of thze models. However we will not assume that here and we will use the theory of Okounkov bodies to independently show that this volume is given by a limit.
\end{remark}
\subsection{Okounkov bodies for adelic divisors}
 We recall the valuation function crucial in the definition of Okounkov bodies. Note that as we remarked at the end of section 1.1, every element of $\overline{H^0}(U,O(\overline{D}))$ can be identified as a global section of $O(D)$ on  $U$ by restricting the effectivity relation $\text{div}(f)+\overline{D}\ge 0$ to $U$. Now we fix a closed regular point $x\in U(K)$ and consider any local trivialisation $s_0$ of $O(D)$ around $x$. Then every element $s\in H^0(U,\overline{D})\subseteq H^0(U,O(D))$ induces a regular function by $f=\frac{s}{s_0}$ around $x$ and hence an element in the completion $\widehat{O_{U,x}}\cong K[[x_1\ldots x_d]]$ where $d$ is the dimension of $U$ and the second congruence follows from the regularity of $x$. Then we define a valuation like function denoted by $\text{ord}$ as follows:
$$\nu_x(f)=\text{min}\{\alpha\in\mathbb{N}^d\mid f=\sum a_{\alpha}\text{x}^{\alpha}\ \text{in}\ \widehat{O_{U,x}},\ a_{\alpha}\neq 0\}$$
where the minimum is taken with respect to the lexicographic order on the variables $x_1\ldots x_d$ and this function is independent of the choice of $s_0$. Now the choice of a flag $x=Y_0\subset Y_1\subset\ldots Y_d=U$ centered at $x$ gives a choice of variables $x_1\ldots x_d$ as above and hence yields a valuation function $\nu_x$ on $\overline{H^0}(U,\overline{D})$. Note that the sub-spaces $\overline{H^0}(U,m\overline{D})$ are finite dimensional and induces a graded linear series $\{V_m\subseteq H(U,mD)\}$ in the sense of section 1.3 of \cite{lazarsfeld2008convex}. Hence we can define the semi-groups and convex bodies similarly
\begin{definition}
\label{def:okounkov}
Suppose we have the adelic divisor $\overline{D}$. Then we can define the semi-group 
\[\Gamma(\overline{D})=\{(\alpha,m)\in \mathbb{N}^{d+1}\mid \alpha=\nu_x(s)\ \text{for some}\ s\in\overline{H^0}(U,m\overline{D})\}\]
We further define $\Gamma(\overline{D})_m=\Gamma(\overline{D})\cap (\mathbb{N}^d\times \{m\})$. Finally we define the associated \emph{Okounkov body} of $\overline{D}$ as
\[\Delta(\overline{D})=\text{closed convex hull}(\cup_m\frac{1}{m}\cdot \Gamma(\overline{D})_m)=\Sigma(\Gamma(\overline{D}))\cap(\mathbb{R}^d\times\{1\})\]
where $\Sigma(\cdot)$ denotes taking the closed convex cone in the ambient Euclidean space.
\end{definition}

We are going to derive required properties of $\Gamma(\overline{D})$ and $\Delta(\overline{D})$ with the goal of relating its volume to the volume of adelic line bundles.
We begin by showing that eventually the models are big when perturbed a little by the boundary divisor $D_0$ provided $\overline{D}$ is big. This is immediate if we assume Proposition 5.2.1 of \cite{yuan2021adelic}. However even without the full strength of the result, we have the following lemma:
\begin{lemma}
\label{lemma:bigmod}
Suppose $\overline{D}$ is a big adelic divisor given by models $\{X_i,D_i\}$ as above with boundary divisor $D_0$. Then for $j>>0$, $D_j-q_jD_0$ (and hence $D_j$) is a big $\mathbb{Q}$-divisor on $X_j$. In particular, we deduce that there exists a $r_0$ such that $H^0(U,r\overline{D})\neq \{0\}$ for all $r>r_0$.
\end{lemma}
\begin{proof}
We are going to use Fujita approximation (\cite{Fujita1994ApproximatingZD}) for $\mathbb{Q}$-divisors on projective models. Note that the RHS of the inclusions in Lemma \ref{lemma:incl} gives us that $\text{vol}(D_j+q_jD_0)\ge\widehat{\text{vol}}(\overline{D})$. Hence for $\epsilon_i>0$, we can find by Fujita approximation an ample $\mathbb{Q}$-divisor $A_j$ on a birational modification $\pi\colon X'_j\rightarrow X_j$ such that $\pi^*(D_j+q_jD_0)\ge A_j$ and $\text{vol}(A_j)\ge\text{vol}(D_j+q_jD_0)-\epsilon_j$. Then consider the $\mathbb{Q}$-divisor $A_j-2q_jD_0\le D_j-q_jD_0$ where we consider this effectivity relation in $X_j'$ by pulling back both $D_0$ and $D_j$ to $X_j'$ and we omit the notaions of pull-backs. Write $D_0=A-B$ where $A$ and $B$ are nef effective $\mathbb{Q}$-divisors in $X_0$. Then we have
\[\text{vol}(D_j-q_jD_0)\ge\text{vol}(A_j-2q_jD_0)=\text{vol}(A_j+2q_jB-2q_jA)\]\[\ge (A_j+2q_jB)^d-2dq_j(A_j+2q_jB)^{d-1}\cdot A\ge\ A_j^d-2dq_j(A_j+2q_jB)^{d-1}A\ge\]\[\text{vol}(D_j+q_jD_0)-\epsilon_j-2dq_j(A_j+2q_jB)^{d-1}\cdot A\]
Here in the second inequality we have used Siu's inequality to the nef divisors $A_j+2q_jB$ and $2q_jA$ since both $A$ and $B$ were nef in $X_0$ and nefness is preserved under bi-rational pull-backs whereas $A_j$ is ample in $X_j'$, in the third inequality we have used that $A_j$ is nef and $B$ is nef and effective and in the last one we have used $A_j^d=\text{vol}(A_j)\ge\text{vol}(D_j+q_jD_0)-\epsilon_j$.
Now choosing $\epsilon_j\to 0$ as $j\to\infty$ and suppose we can  choose a nef model divisor $N$ such that $A_j+2q_j\pi_j^*B\le \pi_j^*N$ for all $j$. Then we get that 
$$\text{vol}(D_j-q_jD_0)\ge\text{vol}(D_j+q_jD_0)-2dq_jM-\epsilon_j$$
where $M=N^{d-1}A$ is a fixed number independent of $j$.  Noting that both $\epsilon_j$ and $q_j$ go to $0$ as $j\to\infty$ and noting that $\text{vol}(D_j+q_jD_0)\ge \widehat{\text{vol}}(\overline{D})>0$ is bounded from below independently of $j$, the above inequality shows that for large enough $j$, $\text{vol}(D_j-q_jD_0)>0$ which finishes the claim.\\
Hence we are reduced to showing that there exists a model nef divisor $N$ in $\text{Div}(U,k)_{\text{mod}}$ such that $\pi_j^*N\ge A_j+2q_j\pi_j^*B$ for all $j$. To this end choose a positive integer $r$ such that $r>q_j$ for all $j$. Then consider the divisor $D_1+2rD_0+2rB$ in $X_0$. By Serre's finiteness there is a nef divisor $N$ on $X_0$ such that $N\ge D_1+2rD_0+2rB$. Then since $r>q_1>0$ and $D_0$ is effective, we get that $N\ge D_1+q_1D_0+q_jD_0+2q_jB$ for all $j$. But note that we have the effectivity relation $D_1+q_1D_0\ge D_j$ and hence we conclude $N\ge D_j+q_jD_0+2q_jB$. Since we have the effectivity $\pi_j^*(D_j+q_jD_0)\ge A_j$, pulling back by $\pi_j$ we deduce $\pi_j^*N\ge\pi_j^*(D_j+q_jD_0)+2q_j\pi_j^*B\ge A_j+2q_j\pi_j^*B$ as required. 
\end{proof}
From now on onwards thanks to the previous lemma, we fix  once and for all a $j$ such that $D_k-q_kD_0$ is big for all $k\ge j$.
 The first result we want to state is the boundedness of $\Delta(\overline{D})$ where we use the similar result for integral divisors on projective varieties from \cite{lazarsfeld2008convex} to obtain our claim.\\
 We start with a sequence of inclusions.
\begin{lemma}
\label{lemma:easy}
We have a sequence of inclusions 
$$\Gamma(D_j-q_jD_0)_k\subseteq\Gamma(\overline{D})_k\subseteq \Gamma(D_j+q_jD_0)_k$$
for all positive integers $k$
and hence as a consequence 
$$\Delta(D_j-q_jD_0)\subseteq \Delta(\overline{D})\subseteq\Delta(D_j+q_jD_0)$$

\end{lemma}
\begin{proof}
First note that it makes sense to have $\Gamma(\cdot)$ and $\Delta(\cdot)$ in the right and left extremities above even though the arguments are $\mathbb{Q}$-divisors by just viewing them as model adelic divisors in $\text{Div}(U,k)_{\text{mod}}$.
The first  sequence of inclusions then follow easily from the set of injective maps in \ref{lemma:incl} and noting that the construction of $\nu_x$ is local. The second set of inclusions then easily follows from definition of a closed convex hull generated by subsets.
\end{proof}
Finally we can state the boundedness result that we wanted to obtain.
\begin{lemma}
\label{lemma:boundadel}
The subset $\Delta(\overline{D})$ is compact convex subset of $\mathbb{R}^d$.
\end{lemma}
\begin{proof}
The said subset is already closed and convex. Hence it is enough to prove that it is bounded. Note that $R=D_j+q_jD_0$ is a $\mathbb{Q}$-divisor on $X_i$ and hence there is an integer $t$ such that $tR$ is an integral Cartier divisor.  Note that from the RHS of the set of inclusions \ref{lemma:incl} we conclude that for any section $s\in\overline{H^0}(U,kt\overline{D})$ induces a section $s'\in H^0(X_i,ktR)'=H^0(U,kt\overline{R})$ and both of these have the same valuation vector. Hence we get that $\Gamma(t\overline{D})\subseteq \Gamma(tR)$ where the RHS is well defined as $tR$ is an integral Cartier divisor which in turn yields by construction that $\Delta(t\overline{D})\subseteq \Delta(tR)$. On the other hand we have $\Gamma(\overline{D})\subseteq \frac{1}{t}\cdot \Gamma(t\overline{D})$ and hence by construction we get $\Delta(\overline{D})\subseteq \frac{1}{t}\cdot\Delta(t\overline{D})$. This readily gives the boundedness as $\Delta(tR)$ is bounded by Lemma 1.10 of \cite{lazarsfeld2008convex} as $tR$ is integral divisor and $X_i$ is projective.
\end{proof}
\begin{remark}
The proof of boundedness for the projective case in \cite{lazarsfeld2008convex} is based on intersecting ample divisors with the flag which gives us the Okounkov construction. It might be interesting to try to give a proof using intersection theory as there is a new intersection theory now with adelic line bundles on quasi-projective varieties. However the notion of \say{adelic ample divisors} are \say{positive} is not immediate to formulate since pull back of ample bundles by birational morphisms is not necessarily ample again and this might arise as a problem.
\end{remark}
\subsection{Volumes of Okounkov bodies}
We want to relate the volume of the Okounkov body $\Delta(\overline{D})$ with the volume of the adelic divisor $\overline{D}$ as defined in \cite{yuan2021adelic}. It will turn out that they are equal (upto scaling) analogous to the projective case. We start with a lemma listing the properties of the $\Gamma(\overline{D})$ which are sufficient to assert the volume equality. We fix a $j$ as in the previous section once again.\\
We begin by recording a result which relates the dimension of the space of global sections with the cardinality of slices of $\Gamma(\overline{D})$. We denote by $\Gamma(\overline{D})_m=\Gamma(\overline{D})\cap (\mathbb{N}^d\times\{m\})$. Then we have 
\begin{lemma}
\label{lemma:graded}
We have $\#\Gamma_m=\emph{dim}_K(\overline{H^0}(U,m\overline{D}))$
\end{lemma}
\begin{proof}
The claim immediately follows from Lemma 1.4 of \cite{lazarsfeld2008convex} by taking $W=\overline{H^0}(U,m\overline{D})$ and noting that $W$ is finite dimensional from Lemma 5.1.7 in \cite{yuan2021adelic}.
\end{proof}
Next we want to naturally extend the notion of Okounkov bodies to $\mathbb{Q}$-adelic line bundles. One necessary property is to show that the construction of $\Delta(\cdot)$ behaves well with taking integral multiples of adelic divisors which is the content of our next lemma.
Note that if we can show $\text{vol}_{\mathbb{R^d}}(\Delta(\overline{D}))=\lim_{m\to\infty}\frac{\#\Gamma_m}{m^d}$, then with the Lemma \ref{lemma:graded} we have that the Euclidean volume of $\Delta(\overline{D})$ is the same as the volume of $\overline{D}$ as defined in Definition \ref{def:vol} upto scaling by $d!$. It turns out that for the above equality to be true, it is enough for $\Gamma(\overline{D})$ to satisfy certain properties which are purely Euclidean geometric in nature. We wish to state and prove them in our main lemma of this section. Before that we prove a property necessary in our next lemma.
\begin{lemma}
\label{lemma:injectamp}
Suppose $\overline{D}$ is a big adelic divisor on a normal quasi-projective variety $U$ given by the sequence of models $\{X_i,D_i\}$ and rationals $\{q_i\to 0\}$ as usual. Then there is a model $X_j$ such that for all ample divisors $\overline{A}$ on $X_j$, there exists a non-zero section $s_0\in H^0(U,m\overline{D}-\overline{A})$ whenever $m$ is a sufficiently large positive integer.
\end{lemma}
\begin{proof}
The idea is to use Kodaira lemma (Proposition 2.2.6 \cite{lazarsfeld2017positivity}) in the projective case on models approximating $\overline{D}$ from below. More preciesely suppose $D_j'=D_j-q_jD_0$. Then as $\{D_j'\}$ is a sequence also representing the big divisor $\overline{D}$, by Lemma \ref{lemma:bigmod} we can find a $j$ such that $D_j'$ is a big divisor. Now applying the Kodaira lemma on the big divisor $D_j'$ on the projective variety $X_j$, we conclude that for all sufficiently large $m$, there exists a non-zero section of $O(mD_j'-\overline{A})$ on $X_j$ and restricting to $U$, we get a non-zero section $s_0\in H^0(U,mD_j'-\overline{A})$ for all sufficiently large $m$. Now the claim follows from noting that the effectivity relation $\overline{D}\ge D_j'$ implies that $H^0(U,mD_j'-\overline{A})\subseteq H^0(U,m\overline{D}-\overline{A})$.
\end{proof}
\begin{lemma}
\label{lemma:fingen}
Suppose $\overline{D}$ is a big adelic divisor on a normal quasi-projective variety $U$ over $K$. Then the convex body $\Gamma(\overline{D})$  satisfies the following properties: 
\begin{enumerate}
    \item $\Gamma_0=\{0\}$
    \item There exist finitely many vectors $(v_i,1)$ spanning a semi-group $B\subseteq \mathbb{N}^{d+1}$ such that $\Gamma(\overline{D})\subseteq B$. 
    \item $\Gamma(\overline{D})$ generates $\mathbb{Z}^{d+1}$ as a group.
\end{enumerate}

\end{lemma}
\begin{proof}
The first point is trivial. For the second point we follow the proof of Lemma 2.2 in \cite{lazarsfeld2008convex}. Denote by $v_i(s)$ the $i$-th co-ordinate in the valuation vector of a section $s$. Note then $v_i(s)\le mb$ for some large constant $b$ and for all non-zero $s\in H^0(U,m\overline{D})$ due to the fact that $\Delta(\overline{D})$ is bounded (Lemma \ref{lemma:boundadel}) as $\Delta(\overline{D})$ contains $\frac{1}{m}\cdot \Gamma(\overline{D})_m$ for each $m\in\mathbb{N}$. Now a basic algebraic calculation easily shows that $\Gamma(\overline{D})$ is contained in the semi-group generated by the finite set of integer vectors $\{(a_i)\mid 0\le a_i\le b\}$ which shows the second point.  Hence it is enough to prove the third point.\\
To this end, choose a model $X_j$ which satisfies the condition of Lemma \ref{lemma:injectamp}. Then choose a very ample divisor $\overline{A}$ on $X_j$ such that there exists sections $\overline{s_i}$ of $O(\overline{A})$ for $i=0,1,\ldots d$ with $v(\overline{s_i})=e_i$ where $v$ is the valuation vector with respect to the chosen flag and $\{e_i\}$ is the standard basis of $\mathbb{R}^d$ for $i=1,\ldots d$ and $e_0$ is the zero vector, as suggested in the beginning of the proof of Lemma 2.2 in \cite{lazarsfeld2008convex}. Restricting these sections give sections $s_i\in H^0(U,\overline{A})$ with $v(s_i)=e_i$. Now thanks to Lemma \ref{lemma:injectamp} and our choice of $X_j$, we can find non-zero sections $t_i\in H^0(U,(m_0+i)\overline{D}-\overline{A})$ for $i=0,1$ with valuation vectors $v(s_i)=f_i$. Then clearly we find non-zero sections $s_i'=s_i\otimes t_0\in H^0(U,m_0\overline{D})$ and $s_0''=s_0\otimes t_1\in H^0((m+1)\overline{D})$ with valuation vectors $v(s_i')=(f_0+e_i)$ for $i=0,\ldots d$ and $v(s_i'')=f_1$. Hence $\Gamma(\overline{D})$ contains the vectors $(f_0,m_0)$, $(f_0+e_i,m_0)$ for $i=1\ldots d$ and $(f_1,m_0+1)$. Then it clearly shows that $\Gamma(\overline{D})$ generated $\mathbb{Z}^d$ as a group and finishes the proof.
\end{proof}
We are ready to state the first main theorem of this chapter.
\begin{theorem}
\label{thm:okoun}
Suppose we have a big adelic divisor $\overline{D}$ on a normal quasi-projective variety $U$ and suppose $\Delta(\overline{D})$ is the Okounkov body associated to $\overline{D}$ as constructed above. Furthermore suppose $\widehat{\emph{vol}}(\overline{D})$ be the adelic volume defined in section 5 of \cite{yuan2021adelic}. Then we have 
$$\emph{vol}_{\mathbb{R}^d}(\Delta(\overline{D}))=\lim_{m\to\infty}\frac{\#\Gamma_m}{m^d}=\lim_{m\to\infty}\frac{\emph{dim}_K(\overline{H^0}(U,m\overline{D}))}{m^d}=\frac{1}{d!}\cdot \widehat{\emph{vol}}(\overline{D})$$
\end{theorem}
\begin{proof}
With Lemma \ref{lemma:fingen} and by basic arguments of euclidean and convex geometry as indicated in the proof of Proposition 2.1 of \cite{lazarsfeld2008convex}, we get that
\begin{equation}
   \label{eqn:wawa}
   \text{vol}_{\mathbb{R^d}}(\Delta(\overline{D}))=\lim_{m\to\infty}\frac{\#\Gamma_{m}}{m^d}=\lim_{m\to\infty}\frac{\text{dim}_K(H^0(U,m\overline{D})}{m^d}
\end{equation}
exists which clearly gives the claim.
\end{proof}
\begin{remark}
Note that the above theorem also proves that the $\limsup$ in the definition of $\widehat{\text{vol}}(\overline{D})$ is actually given by a limit directly from convex geometric properties of the Okounkov bodies which is essentially the content of the first part of Theorem 5.2.1 of \cite{yuan2021adelic}.
\end{remark}
We end this section by showing that the construction of Okounkov body is homogenous with respect to scaling. 
\begin{lemma}
\label{lemma:homogen}
Suppose $\overline{D}$ is a big adelic divisor on a normal quasi-projective variety $U$. Then 
\[\Delta(t\overline{D})=t\cdot \Delta(\overline{D})\]
for all positive integers $t$. Hence we can naturally extend the construction of $\Delta(\cdot)$ to big adelic $\mathbb{Q}$-divisors.
\end{lemma}
\begin{proof}
We choose an integer $r_0$ such that $H^0(U,r\overline{D})\neq \{0\}$ for all $r>r_0$. We can always do this as we assumed $\overline{D}$ is big (as explained in the proof of Lemma \ref{lemma:fingen}). Next choose $q_0>0$ such that $q_0t-(t+r_0)>r_0$ for all $t$. Then for all $r_0+1\le r\le r_0+t$ we can find non-zero sections $s_r\in H^0(U,r\overline{D})$ and $t_r\in H^0(U,(q_0t-r)\overline{D})$ which gives inclusions 
\[H^0(U,mt\overline{D})\xhookrightarrow{\otimes s_r}H^0(U,(mt+r)\overline{D})\xhookrightarrow{\otimes t_r}H^0(U,(m+q_0)t\overline{D})\]
which gives the corresponding inclusion of the graded semi-groups
\[\Gamma(t\overline{D})_m+e_r+f_r\subseteq\Gamma(\overline{D})_{mt+r}+f_r\subseteq\Gamma(t\overline{D})_{m+q_0}\]
where $e_r=v(s_r)$ and $f_r=v(t_r)$. Now recalling the construction of $\Delta(\cdot)$ and letting $m\to\infty$ we get 
\[\Delta(t\overline{D})\subseteq t\cdot \Delta(\overline{D})\subseteq\Delta(t\overline{D})\]
which clearly finishes our proof.
\end{proof}
\begin{remark}
Note that this homogeneity allows us to define Okounkov bodies for adelic $\mathbb{Q}$-divisors by passing to integral multiples and hence conclude that adelic volumes are homogenous for big divisors as in the projective case.
\end{remark}
\subsection{Variation of Okounkov bodies}
We fix a normal quasi-projective variety $U$ over $K$ and a big adelic divisor $\overline{D}$ on it. Furthermore  suppose $\overline{E}$ is any adelic divisor on $U$. We will construct a global convex body $\Delta(U)=\Delta(U,\overline{D},\overline{E})\subseteq \mathbb{R}^d\times \mathbb{R}^2$ such that the fiber of this body over a vector $(a_1,a_2)\in \mathbb{Q}^2$ under the projection to $\mathbb{R}^2$ will give us the Okounkov body of the adelic $\mathbb{Q}$-divisor $q_1\overline{D}+q_2\overline{E}$ provided it is big. Furthermore we fix a flag $Y_d\subset \ldots\subset Y_0$ as before.  We are going to follow closely the arguments in Section 4 of \cite{lazarsfeld2008convex}. All constructions are dependent on the choice of the divisor $\overline{D}$ and $\overline{E}$ but we fix them for this section and we omit them in the notation.
We start by defining the semi-group associated to these two adelic divisors.
\begin{definition}
\label{def:globalbody}
Suppose $\overline{D}$ and $\overline{E}$ are as before. We define the graded semi-group $\Gamma(U)$ as
\[\Gamma(U)=\{((v(s),a_1,a_2)\mid a_i\in \mathbb{Z}, s\neq 0\in H^0(U,a_1\overline{D}+a_2\overline{E})\}\]
where $v(\cdot)$ is the valuation corresponding to the chosen flag. Further more we define the global Okounkov body $\Delta(U)$ as 
\[\Delta(U)=\text{closed convex cone}(\Gamma(U))\]
which is a closed convex subset of $\mathbb{R}^d\times \mathbb{R}^2$.
\end{definition}
As in the case with one bundle, we will deduce the properties needed from general properties of convex bodies and graded semi-groups. Before doing that we define certain terms necessary.
\begin{definition}
\label{def:conv}
Suppose we have an additive semi-group $\Gamma$ in $\mathbb{R}^d\times \mathbb{R}^2$. Denote by $P$ the projection from $\mathbb{R}^d\times \mathbb{R}^2$ to $\mathbb{R}^2$ and $\Delta=\Sigma(\Gamma)$ is the closed convex cone generated by $\Gamma$. We define the \emph{support} of $\Delta$, denoted as $\text{Supp}(\Delta)$ to be its image under $P$. It coincides with the closed convex cone in $\mathbb{R}^2$ generated by the image of $\Gamma$ under $P$. Finally given a vector $\vec{a}=(a_1,a_2)\in \mathbb{Z}^2$ we denote 
\[\Gamma_{\mathbb{N}\vec{a}}=\Gamma\cap (\mathbb{N}^r\times \mathbb{N}\vec{a})\]
\[\Delta_{\mathbb{R}\vec{a}}=\Delta\cap(\mathbb{R}^d\times\mathbb{R}\vec{a})\]
Furthermore we denote $\Gamma_{\mathbb{N}\vec{a}}$ as a semi-group inside $\mathbb{N}^d\times\mathbb{N}\vec{a}\cong\mathbb{N}^{d+1}$ and denote the closed convex cone generated by it in $\mathbb{R}^{d+1}$ as $\Sigma(\Gamma_{\mathbb{N}\vec{a}})$.
\end{definition}
With the above definitions we can state our next lemma.
\begin{lemma}
\label{lemma:interior}
Suppose the semi-group $\Gamma$ generates a sub-group of finite index in $\mathbb{Z}^{d+2}$ and suppose $\vec{a}\in\mathbb{N}^2$ such that $\vec{a}\in\emph{int}(\emph{Supp}(\Delta))$. Then we have 
\[\Delta_{\mathbb{R}\vec{a}}=\Sigma(\Gamma_{\mathbb{N}\vec{a}})\]
\end{lemma}
\begin{proof}
The statement and the proof of the Lemma is identical as in Proposition 4.9 of \cite{lazarsfeld2008convex}.
\end{proof}
Next we want to show that the vectors which gives rise to big combinations of the bundles $\overline{D}$ and $\overline{E}$ in fact belong to the interior $\text{int}(\text{Supp}(\Delta))$ which is the content of the next lemma. Note that by passing to rational multiples just as in the projective case, we can similarly define $\mathbb{Q}$-adelic divisors. Furthermore by the remark at the end of the previous section, we can also define Okounkov bodies for $\mathbb{Q}$-adelic divisors which behave homogenously.
\begin{lemma}
\label{lemma:biginter}
Suppose $\vec{a}\in\mathbb{Q}^2$ such that $a_1\overline{D}+a_2\overline{E}$ is a big adelic divisor. Then $\vec{a}\in\emph{int}(\emph{Supp}(\Delta))$.
\end{lemma}
\begin{proof}
We assume that $\text{dim}(U)>0$ as the $0$-dimensional case is degenerate. Clearly it is enough to show the case when $a_i\in \mathbb{Z}$ because $\text{Supp}(\Delta)$ is a cone and scaling sends open sets to open sets. We can assume that both $\overline{D}$ and $\overline{E}$ are given by models $D_i$ and $E_i$ on projective models $X_i$ of $U$ respectively along with a boundary divisor $D_0$ and rationals $q_i\to 0$ as in our usual notation.
We first prove that for any rational $q\in \mathbb{Q}$ such that $\overline{D}+q\overline{E}$ is big, there is an $\epsilon>0$ such that $(1,x)$ is in $\text{Supp}(\Delta)$ for all $x\in(q-\epsilon,q+\epsilon)$. Suppose first that $q>0$. Then note that the sequence of models $S_j=(D_j-q_jD_0)+q(E_j-q_jD_0)$ gives a Cauchy sequence defining $\overline{D}+q\overline{E}$ and hence by Lemma \ref{lemma:bigmod} we get that $S_j$ is big for large enough $j$. Now due to the continuity of the volume function in the projective setting, we can find a rational $0<q<p$ such that $(D_j-q_jD_0)+p(E_j-q_jD_0)$ is big. Now due to the effectivity relation 
\[(D_j-q_jD_0)+p(E_j-q_jE_0)\le \overline{D}+p\overline{E}\]
we deduce that the right hand side above is big. Hence we get that for some positive integer $p_0$, $p_0\cdot(1,p)\in P(\Gamma)$ where $P\colon\mathbb{R}^d\times\mathbb{R}^2\rightarrow \mathbb{R}^2$ is the projection and $\Gamma=\Gamma(\overline{D})$. As $\overline{D}$ is assumed to be big, we also obtain that $r_0\cdot(1,0)\in P(\Gamma)$ as well for some large positive integer $r_0$. As $p_0$ and $r_0$ are positive, it is enough to find an $\epsilon>0$ such that $(1,x)$ is in the convex cone generated by $(1,0)$ and $(1,p)$ for all $x\in(q-\epsilon,q+\epsilon)$ because $\text{Supp}(\Delta)$ is exactly the convex cone generated by $P(\Gamma)$. But clearly $(1,x)$ is in the convex cone generated by $(1,0)$ and $(1,p)$ for all $0<x<p$ which clearly yields the existence of one such $\epsilon$ because $0<q<p$. For the case when $q<0$ we do a similar calculation but with $E_j-q_jD_0$ being replaced by $E_j+q_jD_0$. Finally for the case $q=0$ using similar arguments we can find a rational number $q_0$ such that all the three vectors $p_0\cdot(1,0)$, $p_0\cdot(1,-q_0)$ and $p_0\cdot(1,q_0)$ are in $P(\Gamma)$ for some large positive integer $p_0$. Hence by the above arguments we get that $(1,x)$ is in $\text{Supp}(\Delta)$ for $x\in(-q_0,q_0)$.\\
Next we take any $\vec{a}=(a_1,a_2)$ such that $a_1\overline{D}+a_2\overline{E}$ is big. First suppose $a_1\le 0$. It is easy to see that the sum of two big adelic divisors is again big. Hence adding $(-a_1)\overline{D}$ we conclude that $a_2\overline{E}$ is big. Since the trivial adelic divisor is not big, we conclude that $a_2\neq 0$. Then adding the big adelic divisor $-a_1\overline{D}$ we deduce that $\overline{E}$ (resp. -$\overline{E})$ is big if $a_2> 0$ (resp. $a_2<0)$. Hence in these two cases replacing $\overline{D}$ by $\overline{E}$ or $-\overline{E}$ we are reduced to the case when $a_1>0$ and hence we can assume WLOG that $a_1>0$. In that case scaling by $a_1$ we obtain that $\overline{D}+q\overline{E}$ is big for $q=\frac{a_2}{a_1}$ and by our considerations before we deduce that for some $\epsilon>0$, $(1,x)$ is in the convex cone generated by $P(\Gamma)$ for all $x\in(q-\epsilon,q+\epsilon)$. We assume that $a_2\ge 0$ and the argument for $a_2<0$ will just be the analogue by changing signs. Then for any $\kappa>0$ we have 
\[\frac{a_2-\kappa}{a_1+\kappa}\le\frac{a_2+t_2}{a_1+t_1}\le \frac{a_2+\kappa}{a_1-\kappa}\]
for all $t_1,t_2\in(-\kappa,+\kappa)$. Choose $\kappa>0$ so small that 
\[(\frac{a_2-\kappa}{a_1+\kappa},\frac{a_2+\kappa}{a_1-\kappa})\subset (q-\epsilon,q+\epsilon)\]
and $a_1\pm \kappa>0$
which we can do as $q=\frac{a_2}{a_1}$ and $a_1>0$. Hence by the choice of $\epsilon$ for any $t_1,t_2\in(-\kappa,\kappa)$, the vector $(1,\frac{a_2+t_2}{a_1+t_1})$ and hence $(a_1+t_1,a_2+t_2)$ is in the convex cone generated by $P(\Sigma)$ and hence in $\text{Supp}(\Delta)$ as $a_1+t_1>0$. This clearly shows that $(a_1,a_2)\in\text{int}(\text{Supp}(\Delta))$ and finishes the proof.
\end{proof}
Next to use Lemma \ref{lemma:interior} we have to prove that $\Gamma(U)$ generates a sub-group of finite index in $\mathbb{Z}^{d+2}$ which in particular guarantees that $\text{int}(\text{Supp}(\Delta(U))$ is non-empty. This is going to be the content of our next Lemma.
\begin{lemma}
\label{lemma:nonemptint}
The multi-graded semi-group $\Gamma(U)$ constructed in Definition \ref{def:globalbody} generates $\mathbb{Z}^{d+2}$ as a group.
\end{lemma}
\begin{proof}
Arguing similarly as in the proof of Lemma \ref{lemma:biginter}, as $\overline{D}$ is big, we can find a positive integer $m$ such that $m\overline{D}-\overline{E}$ is big. On the other hand we already know that $\overline{D}$ is big. Note that the semi-groups $\Gamma(m\overline{D}-\overline{E})$ and $\Gamma(\overline{D})$ sit naturally as sub-semigroups of $\Gamma(U)$. Moreover from Lemma \ref{lemma:fingen} we deduce that $\Gamma(\overline{D})$ and $\Gamma(m\overline{D}-\overline{E})$ generate the sub-groups $\mathbb{Z}^d\times\mathbb{Z}\cdot(1,0)$ and $\mathbb{Z}^d\times\mathbb{Z}\cdot(m,-1)$. But the vectors $(1,0)$ and $(m,-1)$ generate $\mathbb{Z}^2$ which clearly shows that $\Gamma(U)$ generates $\mathbb{Z}^{d+2}$ as a group. 
\end{proof}
Finally we are ready to state and prove the main theorem of this section.
\begin{theorem}
\label{theorem:globalbody}
Suppose $\overline{D}$ and $\overline{E}$ be adelic divisors on a normal quasi-projective variety $U$ such that $\overline{D}$ is big. Then there exists a convex body $\Delta(U)=\Delta(U,\overline{D},\overline{E})\subset \mathbb{R}^{d+2}$ with the property that for any $\vec{a}=(a_1,a_2)\in\mathbb{Q}^2$ with $a_1\overline{D}+a_2\overline{E}$ big, we have \[\Delta(a_1\overline{D}+a_2\overline{E})=\Delta(U)\cap(\mathbb{R}^d\times\{\vec{a}\})\]
where $\Delta(a_1\overline{D}+a_2\overline{E})$ is the Okounkov body of $a_1\overline{D}+a_2\overline{E}$ as constructed in Definition \ref{def:okounkov}.
\end{theorem}
\begin{proof}
Clearly it is enough to show when $\vec{a}\in\mathbb{Z}^2$ by homogeneity of Okounkov bodies (Lemma \ref{lemma:homogen}). Note that the semi-group $\Gamma(a_1\overline{D}+a_2\overline{E})$ sits naturally in $\mathbb{N}^d\times \mathbb{N}\cdot\vec{a}\cong \mathbb{N}^{d+1}$ and by construction of $\Delta(\cdot)$ as in Definition \ref{def:okounkov}, we deduce that $\Delta(a_1\overline{D}+a_2\overline{E})=\Sigma(\Gamma(U)_{\mathbb{N}\vec{a}})\cap(\mathbb{R}^d\times\{\vec{a}\})$. By Lemma \ref{lemma:biginter} we get that $\vec{a}\in\text{int}(\text{Supp}(\Delta(U))$ and hence by Lemma \ref{lemma:interior} we have $\Delta(U)_{\mathbb{R}\vec{a}}=\Sigma(\Gamma(U)_{\mathbb{N}\vec{a}})$. Hence we deduce that 
\[\Delta(a_1\overline{D}+a_2\overline{E})=\Sigma(\Gamma(U)_{\mathbb{N}\vec{a}})\cap(\mathbb{R}^d\times\{\vec{a}\})=\Delta(U)_{\mathbb{R}\vec{a}}\times(\mathbb{R}^d\times\{\vec{a}\})=\Delta(U)\cap(\mathbb{R}^d\times\{\vec{a}\})\]
concluding the proof.
\end{proof}
\begin{remark} 
The construction of the Global body $\Delta(U,\overline{D},\overline{E})$ is done here by mimicking the constructions in section 4 of \cite{lazarsfeld2008convex}. However one stark difference is that the Global body constructed in \cite{lazarsfeld2008convex} is independent of the chosen basis of the Neron-Severi group because they work modulo numerical equivalences. However even if there can be a notion of \say{numerical equivalence} in the adelic setting,it is certainly not known if the corresponding Neron-Severi space is finitely generated and hence such a \say{canonical global body} cannot be constructed using similar methods and our $\Delta(U,\overline{D},\overline{E})$ is dependent on the chosen divisors $\overline{D}$ and $\overline{E}$. However our version still gives some interesting corollaries which we shall see next.
\end{remark}
\subsection{Corollaries : Continuity, Fujita approximation and more}
Before going to state our first corollary, we introduce the notion of \emph{Hausdorff distance} which will be the correct metric under which we want to show the convergence of bodies.
\begin{definition}
\label{def:minkdist}
Let $(V,\|\cdot\|)$ be a normed real vector space.
The \emph{Hausdorff distance} between two closed compact subsets $C_1$ and $C_2$ in $V$ is defined as
\[d_H(C_1,C_2)=\text{inf}\{\epsilon>0\mid C_1\subseteq C_2+\epsilon\mathbb{B}, C_2\subseteq C_1+\epsilon\mathbb{B}\}
\]
where $\mathbb{B}$ is the unit ball in $V$ with respect to $||\cdot||$.
\end{definition}
Now we can state our first main corollary.
\begin{corollary}
\label{corol:coonvbod}
Suppose $\overline{D}$ is a big adelic divisor on a normal quasi-projective variety $U$ given by models $\{X_i,D_i\}$ in our usual notation. Then 
\[\lim_{j\to\infty}d_H(\Delta(\overline{D}),\Delta(\overline{D_j}))=0\] where $\overline{D_j}$ is just $D_j$ looked at as a model divisor in $\emph{Div}(U,k)_{\emph{mod}}$. In particular, we have 
\[\widehat{\emph{vol}}(\overline{D})=\lim_{j\to\infty}\emph{vol}(D_j)\]
where $\emph{vol}(\cdot)$ is the classical projective volume considering $D_j$ as a $\mathbb{Q}$-divisor in $X_j$.
\end{corollary}
\begin{proof}
We prove the first claim at first. Begin by noting that the sequence of inclusions 
\[\Delta(\overline{D}-q_j\overline{D_0})\subseteq \Delta(\overline{D_j})\subseteq \Delta(\overline{D}+q_j\overline{D_0})\]
implies that it is enough to show that $d_H(\Delta(\overline{D}-q_jD_0),\Delta(\overline{D}+q_jD_0))\to 0$ as $j\to\infty$. But this immediately follows from Theorem  \ref{theorem:globalbody} taking $\overline{E}=\overline{D_0}$ and Theorem 13 in \cite{Khovanskii2012-sb} noting that $q_j\to 0$ as $j\to\infty$. Now the second claim follows readily from Theorem 7 in \cite{Shephard1965MetricsFS} and the first claim noting that $\text{vol}(D_j)=\widehat{\text{vol}}(\overline{D_j})=d!\cdot \text{vol}(\Delta(\overline{D_j}))$ and $\text{vol}(\overline{D})=d!\cdot \text{vol}(\Delta(\overline{D}))$.
\end{proof}
\begin{remark}
Note that Corollary \ref{corol:coonvbod} and Theorem \ref{thm:okoun} prove Theorem 5.2.1 of \cite{yuan2021adelic} for big adelic divisors independently using convex geometric methods and hence we can deduce all the corollaries of section 5 of \cite{yuan2021adelic} coming from Theorem 5.2.1 for bid divisors which we list next.
\end{remark}
\begin{corollary}[log-concavity]
\label{corollary:cont}
Suppose $\overline{D}_1$ and $\overline{D}_2$ are two effective adelic divisors on a normal quasi-projective variety $U$. Then we have 
\[\widehat{\emph{vol}}(\overline{D}_1+\overline{D}_2)^{\frac{1}{d}}\ge\widehat{\emph{vol}}(\overline{D}_1)^{\frac{1}{d}}+\widehat{\emph{vol}}(\overline{D}_2)^{\frac{1}{d}}\]
where $d=\emph{dim}(U).$

\end{corollary}
\begin{proof}
The statement is trivial if one of the divisors is not big. When both of them are big, applying Corollary \ref{corol:coonvbod} the problem gets converted into the projective case which is proved in Corollary 4.12 in \cite{lazarsfeld2008convex}.
\end{proof}
\begin{corollary}[Fujita approximation]
\label{corollary:Fujita}
Suppose $\overline{D}$ is a big adelic $\mathbb{Q}$-divisor on a normal quasi-projective variety $U$. Then for any $\epsilon>0$ there exists a normal quasi-projective variety $U'$, a birational morphism $\pi\colon U'\rightarrow U$, a projective model $X'$ of $U'$ and an ample $\mathbb{Q}$-divisor $A'$ on $X'$ such that $\pi^*\overline{D}-A'\ge 0$ in $\widehat{\emph{Div}}(U,K)$ and 
\[\emph{vol}(A')\ge\widehat{\emph{vol}}(\overline{D})-\epsilon\]
where $\emph{vol}(A')$ is the volume of $A'$ as a divisor on $X'$.

\end{corollary}
\begin{proof}
Using the fact that the adelic volume is the limit of its models in Corollary \ref{corol:coonvbod}, the claim gets reduced to the original Fujita approximation which was proved in \cite{Fujita1994ApproximatingZD}.
\end{proof}
Next we come to the final corollary of this section which shows the continuity of the volume function.
\begin{corollary}[continuity]
\label{corol:cont}
Suppose $\overline{D},\overline{M}_1,\ldots \overline{M}_r$ are adelic $\mathbb{Q}$-divisors on a normal quasi-projective variety $U$. Then we have 
\[\lim_{t_1,t_2\ldots t_r\to 0}\widehat{\emph{vol}}(\overline{D}+t_1\overline{M}_1+\ldots t_r\overline{M}_r)=\widehat{\emph{vol}}(\overline{D})\]
where $t_1\ldots t_r$ are rational numbers converging to 0. Furthermore we have $\widehat{\emph{vol}}(\overline{D})=\lim_{j\to\infty}\emph{vol}(D_j)$ for a sequence of model $D_j$ representing $\overline{D}$.
\end{corollary}
\begin{proof}
As in the proof of Theorem 5.2.8 in \cite{yuan2021adelic}, we choose nef model adelic divisors $\overline{M}_i'$ such that $\overline{M}_i'\pm \overline{M}_i\ge 0$ and we set $\overline{M}=\overline{M}_1'+\ldots \overline{M}_r'$. Then it is enough to show that 
\[\lim_{t\to 0}\widehat{\text{vol}}(\overline{D}+t\overline{M})=\widehat{\text{vol}}(\overline{D})\]
as $t$ converges to $0$ over the rationals. First assume that $\overline{D}$ is big. Then from Theorem 13 of \cite{Khovanskii2012-sb} we get 
\[\lim_{t\to 0}d_H(\Delta(\overline{D}+t\overline{M}),\Delta(\overline{D}))=0\]
by taking $\overline{E}=\overline{M}$ in Theorem \ref{theorem:globalbody} since we saw in the proof of Lemma \ref{lemma:biginter} that $\overline{D}+t\overline{M}$ is big for small enough $t$ whenever $\overline{D}$ is big. Now the claim follows  from  Theorem 7 of \cite{Shephard1965MetricsFS}. The second claim is also true when $\overline{D}$ is big thanks to Corollary \ref{corol:coonvbod}. Hence we can assume that $\overline{D}$ is not big. Now suppose the claim does not hold. Then there is a $c>0$ and a sequence of rationals $t_i\to 0$ such that $\widehat{\text{vol}}(\overline{D}+t_i\overline{M})>c$ for all $t_i$. By Corollary \ref{corollary:Fujita} we can choose an ample $\mathbb{Q}$-divisor $A_{t_i}$ on a projective model $X'$ of a birational modification $\pi\colon U'\rightarrow U$ of $U$ such that $\pi^*(\overline{D}+t_i\overline{M})-A_{t_i}\ge 0$ and $\text{vol}(A_{t_i})>c/2$. Then clearly 
\[\widehat{\text{vol}}(\overline{D})\ge \text{vol}(A_{t_i}-t_i\overline{M})\ge A_{t_i}^d-dt_iA_{t_i}^{d-1}\overline{M}\]
where in the second inequality we used the Siu's criterion for model nef divisors $A_{ti}$ and $\overline{M}$. We can bound the intersection number $A_{t_i}^{d-1}\overline{M}$ as in the proof of Theorem 5.2.8 in \cite{yuan2021adelic} to conclude that 
\[\widehat{\text{vol}}(\overline{D})\ge A_{t_i}^d-O(t_i)>c/2-O(t_i)\ \text{as}\ t_i\to 0\]
which clearly contradicts the hypothesis $\widehat{\text{vol}}(\overline{D})=0$ and finishes the proof of the first claim. Furthermore the effectivity relation $\overline{D}_j\le \overline{D}+q_j\overline{D}_0$ shows that $\text{vol}(D_j)\le \widehat{\text{vol}}(\overline{D}+q_j\overline{D}_0)$. Now as $j\to\infty$ we know that $q_j\to 0$ and hence by the first claim $\lim_{j\to\infty}\widehat{\text{vol}}(\overline{D}+q_j\overline{D}_0)=0$ which clearly shows the second claim.
\end{proof}
\section{Augmented base loci and Restricted volumes}
In this chapter we define $\emph{restricted}$ volumes of adelic divisors along a closed sub-variety of a normal quasi-projective variety $U$ over an algebraically closed field $K$. We will define the notion of \emph{augmented base locus} of an adelic divisor. It turns out that the restricted volume can be realised as the volume of an Okounkov body when the the sub-variety is not contained in augmented base locus of the adelic divisor. As a corollary we will deduce that the $\limsup$ defining the restricted volume is actually a limit analogously as in chapter 1 for ordinary volumes. We go on to show that there are global bodies which regulate the variation of restricted volumes along arbitary directions similarly as to ordinary volumes as in chapter 1. Finally as a corollary we will deduce properties analogous to those obtained in chapter one for ordinary volumes.
\subsection{Augmented base locus of an adelic divisor}
In this section, we recall the concepts of base loci and stable base loci of a graded linear series of an adelic line bundle. Using these concepts we introduce the notion of the \emph{augmented base locus} of an adelic divisor $\overline{D}$ in analogy to the projective setting (see \cite{lazarsfeld2008convex} section 2.4). In the projective setting, it is shown that the definition of augmented base locus is independent of the choice of the ample divisor using Serre's finiteness. However as in our setting, model divisors are only defined upto bi-rational pull-backs and ampleness is not preserved under such pull-backs, Serre's finiteness does not work. It turns out that this gap can be fixed using the main theorem due to \cite{birkar} and provides us with a similar independence of choice which will be the main result of this section.
\begin{definition}
\label{def:baseloci}
Suppose $U$ is a normal quasi-projective variety over an algebraically closed field $K$ and suppose $D$ is a divisor. Furthermore suppose $W\subseteq H^0(U,O(D))=H^0(U,D)$ is a finite dimensional sub-space of the space of global sections of $O(D)$. Then we define the \emph{base locus} 
\[\emph{Bs}(W)=\{p\in U\mid s(p)=0\ \text{in}\ \kappa(p)=O_{U,p}/m_{U,p}\ \emph{for all}\ s\in W\}\]
Now suppose we have a graded linear series $W=\{W_m\}$ of $O(D)$. We define the \emph{stable base locus} as 
\[\emph{SB}(W)=\cap_{m\in \mathbb{N}}\emph{Bs}(W_m)\]
Finally suppose $\overline{D}$ is an adelic divisor on $U$. Then it determines graded linear series $W=\{W_m=H^0(U,m\overline{D})\}$ as explained in the beginning of chapter 1. Then we define the \emph{base locus} and \emph{stable base locus} of $\overline{D}$ as 
\[\emph{Bs}(\overline{D})=\emph{Bs}(W_1)\ \emph{and}\ \emph{SB}(\overline{D})=\emph{SB}(W)\]
\end{definition}
\begin{remark}
Note that it is easy to check that the stable base locus $\text{SB}(\overline{D})$ is indeed eventually stable \emph{i.e} there exists an integer $p_0$ such that $\text{SB}(\overline{D})=\text{Bs}(p_0\overline{D})$ by using noetherianity of $U$ just like in the projective case.
\end{remark}
As discussed above we want to show that this above notion is invariant under passing to other model ample divisors. Our next lemma is the main ingredient to show that
\begin{lemma}
\label{lemma:modfree}
Suppose $X_1$ and $X_2$ are two normal projective models of a normal quasi-projective variety $U$ over $K$, $f\colon X_1\rightarrow X_2$ a birational morphism which is an isomorphism over $U$ and $\overline{A}_1. \overline{A}_2$ ample divisors on $X_1$ and $X_2$ respectively. Furthermore suppose $\overline{D}$ is an adelic divisor on $U$. Then for any closed irreducible sub-variety $E$ of $U$,  $E\nsubseteq\emph{Bs}(m_0\overline{D}-\overline{A}_2)$ for some positive integer $m_0$ if and only if $E\nsubseteq\emph{Bs}(n_0\overline{D}-\overline{A}_1)$ for some positive integer $n_0$.
\end{lemma}
\begin{proof}
We first suppose that $E\nsubseteq \text{Bs}(m_0\overline{D}-\overline{A}_2)$.
We denote $f^*\overline{A}_2=\overline{A}_2'$ which is a big nef divisor on $X_1$ as $\overline{A}_2$ was big and nef( being ample) and this notions are invariant under bi-rational pull-backs. Let $\overline{E}$ be the Zariski closure of $E$ in $X_1$. Then clearly $\overline{A}_2'|_{\overline{E}}$ is big and as $\overline{A}_2'$ is also nef, we can deduce from Theorem 1.4 of \cite{birkar} that for large enough integer $s_0$, $\overline{E}$ is not contained in the (projective) stable base locus of $s_0\overline{A}_2'-\overline{A}_1$ since $\overline{A}_1$ is ample in $X_1$. Restricting everything to $U$, we can find a positive integer $p_0$ and section $s'\in H^0(U,s_0p\overline{A}_2-p\overline{A}_1)$ such that $s'$ does not vanish along $E$ whenever $p_0\mid p$. Tensoring by a section of $H^0(U,(p-1)\overline{A}_1)$ non-vanishing on $E$, which we can find as $\overline{A}_1$ is ample, we produce a section $s\in H^0(s_0p\overline{A}_2-\overline{A}_1)$ non-vanishing on $E$ whenever $p_0\mid p$. By hypothesis we can find a section $s_0\in H^0(U,m_0s_0p_0\overline{D}-s_0p_0\overline{A}_2)$ non-vanishing along $E$. Hence picking $p=p_0$ and tensoring $s$ and $s_0$ we produce a section in $H^0(U,m_0s_0p_0\overline{D}-\overline{A}_1)$ which does not vanish identically on $E$ and hence $E\nsubseteq\text{Bs}(n_0\overline{D}-\overline{A}_1)$ and finishes one direction of the claim with $n_0=m_0s_0p_0$.\\
For the other side, suppose $E\nsubseteq\text{Bs}(n_0\overline{D}-\overline{A}_1)$. Hence for every positive integer $p$ we can find a section $s_0\in H^0(U,n_0p\overline{D}-p\overline{A}_1)$ which does not vanish identically on $E$. Now chose $p$ large enough such that $p\overline{A}_1-\overline{A}_2'$ is very ample which we can do by Serre's finiteness theorem on projective varieties because $\overline{A}_1$ is ample on $X_1$. Then chosing a section of $p\overline{A}_1-\overline{A}_2'$ on $X_1$ not vanishing identically on $\overline{E}$ and restricting to $U$, we obtain a section $s_0\in H^0(U,p\overline{A}_1-\overline{A}_2)$ not vanishing identically on $E$ for large enough $p$. Once again tensoring $s$ and $s_0$ we obtain that $E\nsubseteq\text{Bs}(m_0\overline{D}-\overline{A}_2)$ with $m_0=n_0p$ for large enough $p$ and finishes the proof.
\end{proof}
\begin{remark}
The proof of the above lemma follows along similar lines as the independence of the augmented base locus on the choice of the ample divisor is shown in the projective case. However it uses Serre's finiteness theorem which has no known versions in the adelic setting due to non-invariance of ampleness under birational pull-backs. However it turns out the gap in one direction of the proof can be bridged by the main result due to \cite{birkar} as we have shown above and in the other direction we already have Serre finiteness.
\end{remark}
Finally we can deduce the the desired invariance under pull-backs of model ample divisors as a direct corollary of Lemma \ref{lemma:modfree} which we do next.
\begin{corollary}
\label{corol:modind}
Suppose $\overline{D}$ is an adelic divisor on a normal quasi-projective variety $U$ over $K$ and suppose $X_1$ and $X_2$ are two projective models of $E$ with ample divisors $\overline{A}_1$ and $\overline{A}_2$ respectively on them. Then for any closed irreducible sub-variety $E$ of $U$, we have that $E\nsubseteq\emph{Bs}(m_0\overline{D}-\overline{A}_2)$ for some positive integer $m_0$ if and only if $E\nsubseteq\emph{Bs}(n_0\overline{D}-\overline{A}_1)$ for some positive integer $n_0$. In particular the set $B_+(\overline{D},\overline{A})=\cap_{m\in\mathbb{N}}\emph{Bs}(m\overline{D}-\overline{A})$ is independent of the chosen model ample divisor $(X,\overline{A})$.
\end{corollary}
\begin{proof}
Clearly the the second claim follows from the first and the first claim follows directly from Lemma \ref{lemma:modfree} by noting that we can always find a projective model $X$ of $U$ dominating both $X_1$ and $X_2$ via a birational morphism over $U$ and an ample divisor on $X$.
\end{proof}
The above corollary clearly shows what should be the definition of our augmented base locus which we record in the next definition.
\begin{definition}
\label{def:augloc}
Suppose $\overline{D}$ is an adelic divisor on a normal quasi projective variety $U$ over $K$. We define the $\emph{augmented base locus}$ of $\overline{D}$ as $B_+(\overline{D})=\cap_{m\in\mathbb{N}}\emph{Bs}(m\overline{D}-\overline{A})$ for any ample divisor $\overline{A}$ on a projective model of $U$.
\end{definition}
\begin{remark}
Note that the above definition makes sense thanks to Corollary \ref{corol:modind}. It is easy to check that $B_+(m_0\overline{D})=B_+(\overline{D})$ for any positive integer $m_0$ and hence we can define an augmented base locus of an adelic $\mathbb{Q}$-divisor by passing to integral multiples.
\end{remark}
We end this section with a lemma which will be necessary later to show that the Okounkov bodies of restricted linear series behave nicely when the sub-variety is not contained in the augmented base locus.
\begin{corollary}
\label{corol:consec}
Suppose $\overline{D}$ is an adelic divisor on a normal quasi-projective variety $U$ over $K$ and suppose $E$ is a closed irreducible sub-variety with $E\nsubseteq B_+(\overline{D})$. Then there exist a projective model $X$ such that for any ample divisor $\overline{A}$ on $X$, there exists sections $s_i\in H^0(U,(m_0+i)\overline{D}-p_i\overline{A})$ not vanishing identically on $E$ for some positive integers $m_0,p_0,p_1$ and $i=0,1$.

\end{corollary}
\begin{proof}
Suppose $\overline{D}$ is given a sequence of models $\{X_i,D_i\}$ and rationals $q_i\to 0$ as usual and let $X=X_1$. Then as $E\nsubseteq B_+(\overline{D})$, for any ample divisor $\overline{A}$ on $X_1$ we can assume that $E\nsubseteq\text{Bs}(n_0\overline{D}-\overline{A})$ for some $n_0\in\mathbb{N}$ and hence we can produce a section $s_0\in H^0(U,2n_0p\overline{D}-2p\overline{A})$ not vanishing identically on $E$ for every positive integer $p$. Choose $p$ so large that $D_1'+p\overline{A}$ is very ample where $D_1'=D_1-q_1D_0$ and choose a section $s'\in H^0(U,D_1'+p\overline{A})$ which does not vanish identically on $E$. Then tensoring $s_0$ and $s'$ we get a section $s_1\in H^0(U,2n_0p\overline{D}+D_1'-p\overline{A})\subseteq H^0(U,(2n_0p+1)\overline{D}-p\overline{A})$ where the inclusion follows from the effectivity relation $D_1'\le \overline{D}$. Clearly $s_0$ and $s_1$ satisfy the claim with $m_0=2n_0p$, $p_1=2p$ and $p_2=p$.
\end{proof}
\subsection{Restricted volumes}
In this section, we define the restricted volume of an adelic divisor along a closed sub-variety $E$ of $U$ in analogy to the projective setting. Then we go on to show that if $E$ is such an irreducible  closed sub-variety with $E\nsubseteq B_+(\overline{D})$, then this restricted volume can be realised as the volume of an Okounkov body calculated with respect to a suitable flag dominated by $E$. Much in the spirit of Theorem \ref{thm:okoun} we deduce that the $\limsup$ defining the restricted volume is actually a limit.\\
Suppose we have an irreducible closed sub-variety $E\xhookrightarrow{i}U$ embedding in $U$ via the closed immersion $i$. Then as explained in sub-section 5.2.2 of \cite{yuan2021adelic} we can consider the pullback of the adelic line bundle $O(\overline{D})$ by $i$ which we denote as the \emph{restriction} of $O(\overline{D})$ to $E$ and denote as $O(\overline{D})|_E$. We recall that this line bundle is given by the datum $\{E_i,O(D_i)|_{E_i}\}$ where $D_i$ are the ,models defining $\overline{D}$ and $E_i$ are the Zariski closures of $E$ in the projective models $X_i$ of $U$. Then there is a restriction map of vector spaces on the space of global sections 
\[H^0(U,O(\overline{D}))\xrightarrow{\text{restr}}H^0(E,O(\overline{D})|_E)\]
and we denote the image of this map by $H^0(U|E,O_E(\overline{D}))$
obtained by just restriction maps on sections model wise. This lets us define the notion of restricted volume.
\begin{definition}
\label{def:restrvol}
Suppose $E$ is a closed irreducible sub-variety of a normal quasi-projective variety $U$ over $K$ and $\overline{D}$ be an adelic divisor on $U$. Then we define the $\emph{restricted volume}$ of $\overline{D}$ along $E$ as 
\[\widehat{\emph{vol}}_{U|E}(\overline{D})=\limsup_{m\to\infty}\frac{\text{dim}_K(H^0(U|E,O_E(m\overline{D}))}{m^d/d!}\]
where $d=\text{dim}(E)$.
\end{definition}
We can view the finite-dimensional vector spaces $W_m=H^0(U|E,O_E(m\overline{D}))$ as a graded linear sub-series of $H^0(E,O(m\overline{D})|_E)\subseteq H^0(E,O(mD)|_E)$. And hence if we can fix a flag in $E$, we can construct an Okounkov body corresponding to $\{W_m\}$ as indicated in section 1 of \cite{lazarsfeld2008convex}.\\
Now given a closed sub-variety $E$ in $U$, we fix a flag $Y_0\subset Y_1\ldots\subset Y_d=E$ in $E$ where $\text{dim}(E)=d$. Note that in any projective model of $U$ this flag induces a canonical partial flag contained in the closure $E_j$ of $E$ in $U$ by taking closures we obtain a flag in the model such that the (partial) valuation of a global section of some bundle with respect to this on the model is the same after restricting to $U$ and evaluating w.r.t to the flag $Y_0\subset Y_1\ldots\subset Y_d$ and we always take this flag to calculate valuation vectors in the projective models. We fix this flag to calculate the Okounkov body of the linear series $\{W_m\}$. Then we have the notions of the graded semi-group $\Gamma_{U|E}(\overline{D})\subseteq \mathbb{N}^{d+1}$ and the Okounkov body $\Delta_{U|E}(\overline{D})\subseteq \mathbb{R}^d$. As in chapter 1, we also define $\Gamma_{U|E}(\overline{D})_m$ to be the fiber of the graded semi-group over the positive integer $m$. Next we show that when $E\nsubseteq B_+(\overline{D})$, then the Okounkov body behaves nicely in the sense of satisfying properties analogous to Lemma \ref{lemma:fingen}.
\begin{lemma}
\label{lemma:restrfingen}
Suppose $\overline{D}$ is a adelic divisor on a normal quasi-projective variety $U$ over $K$. Furthermore suppose $E$ is a closed irreducible sub-variety of $U$ such that $E\nsubseteq B_+(\overline{D})$. Then the graded semi-group $\Gamma_{U|E}(\overline{D})$ satisfies the following properties
\begin{enumerate}
    \item $\Gamma_{U|E}(\overline{D})_0=\{0\}$
    \item There exists finitely many vectors $(v_i,1)$ spanning a semi-group $B\subseteq \mathbb{N}^{d+1}$ such that \break $\Gamma_{U|E}(\overline{D})\subseteq B$.
    \item $\Gamma_{U|E}(\overline{D})$ generates $\mathbb{Z}^{d+1}$ as a group.
\end{enumerate}
\end{lemma}
\begin{remark}
Note that in analogy to Lemma \ref{lemma:fingen} it is desirable that $\overline{D}$ is big in the above lemma. However since we assume that $E\nsubseteq B_+(\overline{D})$, by Definition \ref{def:augloc} we already have a non-zero section $s\in H^0(m\overline{D}-\overline{A})$ for some model ample divisor $\overline{A}$ on a projective model $X$ of $U$. Hence we have the inclusion 
\[H^0(U,n\overline{A})\xhookrightarrow{s^{\otimes n}}H^0(U,mn\overline{D})\]
for all positive integers $n$
which shows that $\widehat{\text{vol}}(m\overline{D})\ge \widehat{\text{vol}}(\overline{A})>0$
and hence $\overline{D}$ is big. In other words the assumption $E\nsubseteq B_+(\overline{D})$ already implies that $\overline{D}$ is big.
\end{remark}
\begin{proof}
Suppose $\overline{D}$ is given by the sequence of models $\{X_i,D_i\}$ and rationals $q_i\to 0$ as usual. Note that as in the proof of Lemma \ref{lemma:fingen}, the first point is trivial and the second point can be deduced once we know that the vectors of $\Gamma_{U|E}(\overline{D})_m$ are bounded by $mb$ for some large positive constant $b$ as explained in the proof of Lemma 2.2 in \cite{lazarsfeld2008convex}. In other words we need to show that restricted graded series satisifes the condition (A) as defined in Definition 2.4 of \cite{lazarsfeld2008convex}. Note that the effectivity relation $\overline{D}\le D_j+q_jD_0$ implies the inclusion $\Gamma_{U|E}(\overline{D})_m\subset \Gamma_{U|E}(\overline{D_j+q_jD_0})_m$ for all positive integers $m$. The right hand side is the same as the graded semi-group of $D_j+q_jD_0$ viewed as a $\mathbb{Q}$-divisor on the projective variety $X_j$ calculated with closures of our flag on $E$ and hence by the footnote on page 803 of \cite{lazarsfeld2008convex}, we conclude that $\Gamma_{U|E}(\overline{D_j+q_jD_0})_m$ satisifes condition (A) which clearly shows the second point as $\Gamma_{U|E}(\overline{D})_m$ is a subset. Hence we just need to show the third point.\\
We argue as in the proof of Lemma \ref{lemma:fingen}. Choose a model very ample divisor $\overline{A}$ on $X_j$ such that it has sections $\overline{s_i}$ on $X_j$ with $v(\overline{s_i})=(e_i)$ for $i=0\ldots d$ where $e_0$ is the zero vector, $\{e_i\}$ is the standard basis of $\mathbb{R}^{d}$ for $i=1,\ldots d$ and $v(\cdot)$ is the valuation corresponding to the closures in $X_j$ of the chosen flag in $E$. We can always do this as $\overline{A}$ is chosen very ample and hence the restriction $\overline{A}|_{E_j}$ is very ample where $E_j$ is the closure of $E$ in $X_j$, as explained in proof of Lemma 2.2 in \cite{lazarsfeld2008convex}. Restricting to $U$ gives sections $s_i\in H^0(U,\overline{A})$ with the same valuation vectors. Note that then for all positive integers $p$, by appropriately tensoring these sections we can also find sections $s_{ip}=s_0^{\otimes p-1}\otimes s_i\in H^0(U,p\overline{A})$ such that $v(s_{ip})=(e_i)$. Then by restricting to $E$, we get non-zero sections $s_{ip}|_E\in H^0(U|E,O_E(p\overline{A}))$ with $v(s_{ip}|_E)=e_i$. Now using Corollary \ref{corol:consec} we can find positive integers $m_0,p_0,p_1$ and sections $t_0,t_1$ (recalling them $t_i$ for notational convenience) satisfying the properties stated in the corollary. Restricting $t_i$'s to $E$ we get non-zero sections $t_i|_E\in H^0(U|E,O_E((m_0+i)\overline{D}-p_i\overline{A}))$ and suppose $v(t_i|_E)=f_i$ for $i=0,1$. Then arguing like in the proof Lemma \ref{lemma:fingen} by tensoring $s_{ip}|_E$'s with $t_i|_E$'s we conclude that the vectors $(f_0,m_0)$, $(f_0+e_i,m_0)$ and $(f_1,m_0+1)$ all belong to $\Gamma_{U|E}(\overline{D})$ which clearly completes the proof.
\end{proof}
Then arguing just like in chapter 1, we deduce the main theorem of this section which we state next.
\begin{theorem}
\label{theorem:restrvol}
Suppose $\overline{D}$ is an adelic divisor on a normal quasi-projective variety $U$ over $K$. Furthermore suppose $E$ is a closed irreducible sub-variety of $U$ such that $E\nsubseteq B_+(\overline{D})$. Then we have 
\[\emph{vol}_{\mathbb{R}^d}(\Delta_{U|E}(\overline{D}))=\lim_{m\to\infty}\frac{\#\Gamma_{U|E}(\overline{D})_m}{m^d}=\lim_{m\to\infty}\frac{\emph{dim}_K(H^0(U|E,O_E(m\overline{D})))}{m^d}=\frac{1}{d!}\cdot \widehat{\emph{vol}}_{U|E}(\overline{D})\]
where $\emph{dim}(U)=d$
\end{theorem}
We end this section with a homogeneity property analogous to Lemma \ref{lemma:homogen}. Before going to that we obtain a crucial property needed for the homogeneity.
\begin{lemma}
\label{lemma:eventuallyfree}
Suppose $\overline{D}$ is an adelic divisor on a normal quasi-projective variety $U$ over $K$ and $E$ is a closed sub-variety with $E\nsubseteq B_+(\overline{D})$. Then there exists an integer $r_0$ such that $H^0(U|E,O_E(r\overline{D}))\neq \{0\}$ for all positive integers $r>r_0$.
\end{lemma}
\begin{proof}
By Corollary \ref{corol:consec} there exist non-zero sections $s_i\in H^0(U,(m+i)\overline{D})$ which do not vanish identically on $E$ for $i=0,1$ by tensoring with sections of $p_i\overline{A}$ which do not vanish identically on $E$ which exists as $\overline{A}$ can be assumed very ample. Then for all $r\ge m^2$, write it as $r=a_rm+b_r$ for non-negative integers $a_r\ge m$ and $0\le b_r\le m-1< a_r$. Then note that $s_0^{\otimes a_r-b_r}\otimes s_1^{b_r}$ is a section of $H^0(U,r\overline{D})$ which does not vanish identically on $E$ which clearly finishes the claim with $r_0=m^2-1$
\end{proof}
\begin{lemma}
\label{lemma:restrhomogen}
Suppose $\overline{D}$ is an adelic divisor on a normal quasi-projective variety $U$ over $K$. Then for any closed irreducible sub-variety $E$ of $U$ with $E\nsubseteq B_+(\overline{D})$, we have
\[\Delta_{U|E}(t\overline{D})=t\cdot \Delta_{U|E}(\overline{D})\]
for all positive integers $t$. Hence we can naturally extend the construction of $\Delta_{U|E}(\cdot)$ to big adelic $\mathbb{Q}$-divisors.
\end{lemma}
\begin{proof}
We choose an integer $r_0$ such that $H^0(U|E,O_E(r\overline{D}))\neq \{0\}$ for all $r>r_0$ thanks to Lemma \ref{lemma:eventuallyfree}. Next choose $q_0>0$ such that $q_0t-(t+r_0)>r_0$ for all positive integers $t$. Then for all $r_0+1\le r\le r_0+t$ we can find non-zero sections $s_r\in H^0(U|E,O_E(r\overline{D}))$ and $t_r\in H^0(U|E,O_E((q_0t-r)\overline{D}))$ which gives inclusions 
\[H^0(U|E,O_E(mt\overline{D}))\xhookrightarrow{\otimes s_r}H^0(U|E,O_E((mt+r)\overline{D}))\xhookrightarrow{\otimes t_r}H^0(U|E,O_E((m+q_0)t\overline{D}))\]
which gives the corresponding inclusion of the graded semi-groups
\[\Gamma_{U|E}(t\overline{D})_m+e_r+f_r\subseteq\Gamma_{U|E}(\overline{D})_{mt+r}+f_r\subseteq\Gamma_{U|E}(t\overline{D})_{m+q_0}\]
where $e_r=v(s_r)$ and $f_r=v(t_r)$. Now recalling the construction of $\Delta_{U|E}(\cdot)$ and letting $m\to\infty$ we get 
\[\Delta_{U|E}(t\overline{D})\subseteq t\cdot \Delta_{U|E}(\overline{D})\subseteq\Delta_{U|E}(t\overline{D})\]
which clearly finishes our proof.
\end{proof}
\subsection{Variation of bodies for restricted volumes}
In this section, we construct global bodies whose fibers give the Okounkov bodies $\Delta_{U|F}(\cdot)$ for a ''sufficiently general'' closed sub-variety $F$ of $U$ much in analogy with Theorem \ref{theorem:globalbody}. Most of the constructions follow analogously as in the global case. The crucial point that we need to show is that given a fixed irreducible sub-variety $F$, the set of divisors $\overline{D}$ with $F\nsubseteq B_+(\overline{D})$ is in the interior of the support of the global body as was shown in Lemma \ref{lemma:biginter}. Most of the other arguments will follow identically as in section 5 of chapter 1. However for sake of clarity we will anyway repeat some constructions. We fix a flag $Y_0\subset\ldots Y_d=F$ as explained in the previous section and all calculations of Okounkov bodies is with respect to this flag.\\
\begin{definition}
\label{def:restrglobalbody}
Suppose $\overline{D}$ and $\overline{E}$ are two adelic divisors on a normal quasi-projective variety $U$ over $K$. Given a closed irreducible sub-variety $F$ of $U$ with $F\nsubseteq B_+(\overline{D})$, we define the graded semi-group $\Gamma_{U|E}(E)$ as
\[\Gamma_{U|F}(F)=\{((v(s),a_1,a_2)\mid a_i\in \mathbb{Z}, s\neq 0\in H^0(U|F,O_F(a_1\overline{D}+a_2\overline{E}))\}\]
where $v(\cdot)$ is the valuation corresponding to the chosen flag. Further more we define the global Okounkov body $\Delta(U)$ as 
\[\Delta_{U|F}(F)=\text{closed convex cone}(\Gamma_{U|F}(F))=\Sigma(\Gamma_{U|F}(F))\]
which is a closed convex subset of $\mathbb{R}^d\times \mathbb{R}^2$.
\end{definition}
\begin{definition}
\label{def:conve}
Suppose we have an additive semi-group $\Gamma$ in $\mathbb{R}^d\times \mathbb{R}^2$. Denote by $P$ the projection from $\mathbb{R}^d\times \mathbb{R}^2$ to $\mathbb{R}^2$ and $\Delta=\Sigma(\Gamma)$ is the closed convex cone generated by $\Gamma$. We define the \emph{support} of $\Delta$, denoted as $\text{Supp}(\Delta)$ to be its image under $P$. It coincides with the closed convex cone in $\mathbb{R}^2$ generated by the image of $\Gamma$ under $P$. Finally given a vector $\vec{a}=(a_1,a_2)\in \mathbb{Z}^2$ we denote 
\[\Gamma_{\mathbb{N}\vec{a}}=\Gamma\cap (\mathbb{N}^r\times \mathbb{N}\vec{a})\]
\[\Delta_{\mathbb{R}\vec{a}}=\Delta\cap(\mathbb{R}^d\times\mathbb{R}\vec{a})\]
Furthermore we denote $\Gamma_{\mathbb{N}\vec{a}}$ as a semi-group inside $\mathbb{N}^d\times\mathbb{N}\vec{a}=\mathbb{N}^{d+1}$ and denote the closed convex cone generated by it in $\mathbb{R}^{d+1}$ as $\Sigma(\Gamma_{\mathbb{N}\vec{a}})$.
\end{definition}
We begin by showing the crucial property of the ''good'' divisors being open.
\begin{lemma}
\label{lemma:restrbiginterior}
Suppose $\overline{D}$ and $\overline{E}$ are two adelic divisors such that $F\nsubseteq B_+(\overline{D})$ and  $F\nsubseteq B_+(\overline{D}+q\overline{E})$ for some $q\in\mathbb{Q}$. Then there is an $\epsilon>0$ such that $(1,x)\in \emph{Supp}(\Delta_{U|F}(F))$ for all $x\in(q-\epsilon,q+\epsilon)$.
\end{lemma}
\begin{proof}
Suppose  both $\overline{D}$ and $\overline{E}$ are given by models $D_i,E_i$ on projective models $X_i$ and rationals $\{q_i\to 0\}$ as usual. Further more we denote $E_1'=D_1-q_1D_0$ and $E_1''=E_1+q_1E_0$. We first consider the case when $q\neq 0$. Then by hypothesis there exists an integer $m_0$ depending on $\overline{A}$ such that $F\nsubseteq \text{Bs}(m_0p\overline{D}+m_0qp\overline{E}-p\overline{A})$ for some very ample divisor $\overline{A}$ on $X_1$ and for all sufficiently divisible integers $p$. Choose $p$ so large that that $E_1'+p\overline{A}$ (resp $-E_1''+p\overline{A}$) is very ample when $q>0$( resp. $q<0$). Then choosing sections in \[H^0(U,2m_0p\overline{D}+2m_0qp\overline{E}-2p\overline{A})\ \text{and}\ H^0(U,p\overline{A}+E_1')\ (\text{resp}\ H^0(U,p\overline{A}-E_1''))\] which do not vanish identically on $F$ and finally tensoring them, we get sections in\break $H^0(U,2m_0p\overline{D}+2m_0qp\overline{E}+E_1'-p\overline{A})$ (resp $H^0(U,2m_0p\overline{D}+2m_0qp\overline{E}-E_1''-p\overline{A}))$ which do not\break vanish identically on $F$. Now the effectivity relation $E_1'\le\overline{E}$ (resp $E_1''\ge \overline{E}$) induces the \break inclusion \[H^0(U,2m_0p\overline{D}+2m_0qp\overline{E}+E_1'-p\overline{A})\subseteq H^0(U,2m_0p\overline{D}+(2m_0qp+1)\overline{E}-p\overline{A})\]\[ (\text{resp.}\ H^0(U,2m_0p\overline{D}+2m_0qp\overline{E}-E_1''-p\overline{A})\subseteq H^0(U,2m_0p\overline{D}+(2m_0qp-1)\overline{E}-p\overline{A}))\] when $q>0$ (resp $q<0$). Hence noting the remark at the end of Definition \ref{def:augloc}, we conclude that $F\nsubseteq B_+(2m_0p\overline{D}+(2m_0qp+1)\overline{D})=B_+(\overline{D}+r\overline{E})$ (resp. $F\nsubseteq B_+(2m_0p\overline{D}+(2m_0qp-1)\overline{D})=B_+(\overline{D}+r\overline{E})$) where $r=q+\frac{1}{2m_0p}>q$ (resp $r=q-\frac{1}{2m_0p}<q$)  when $q>0$ (resp $q<0$). Note that then thanks to Lemma \ref{lemma:eventuallyfree} we conclude that for some large integer $p_0$ the points $p_0(1,r),\ p_0(1,0) \in \text{Supp}(\Delta_{U|F}(F))$ as $F\nsubseteq B_+(\overline{D}+r\overline{E})$ and $F\nsubseteq B_+(\overline{D})$. Hence arguing as in the proof of Lemma \ref{lemma:biginter} we obtain the claim. Finally for the case $q=0$ we repeat the arguments above in both positive and negative directions with $E_1'$ and $E_1''$ to obtain such an $\epsilon$.
\end{proof}
As a corollary of the above, we obtain the necessary property which we record next.
\begin{corollary}
\label{corol:restrinterior}
Suppose $\overline{D}$ and $\overline{E}$ are adelic divisors on a normal quasi-projective variety $U$ over $K$ and let $F$ be a closed sub-variety of $U$ with $F\nsubseteq B_+(\overline{D})$. Then for any $\vec{a}=(a_1,a_2)\in \mathbb{Q}^2$ such that $a_1\overline{D}+a_2\overline{E}$ with $F\nsubseteq B_+(a_1\overline {D}+a_2\overline{E})$ we have $\vec{a}\in\emph{int}(\emph{Supp}(\Delta_{U|F}(F)))$

\end{corollary}
\begin{proof}
Due to the homogeneity property in Lemma \ref{lemma:restrhomogen} it is enough to show the claim for $a_i$ integers. First note that if $\overline{D}_1$ and $\overline{D}_2$ are two adelic divisors with $F\nsubseteq B_+(\overline{D}_i)$ for $i=1,2$, then $F\nsubseteq B_+(\overline{D}_1+\overline{D}_2)$. To see this pick a positive integer $m$ such that $F\nsubseteq \text{Bs}(m\overline{D}_i-\overline{A})$ for $i=1,2$ and some ample divisor $\overline{A}$ on some projective model $X$. Then choosing sections on each of the bundles non-vanishing on $F$ and tensoring them, we produce a section in $H^0(U,m(\overline{D}_1+\overline{D}_2)-2\overline{A})$ which does not vanish identically on $E$ which clearly shows that $F\nsubseteq B_+(\overline{D}_1+\overline{D}_2)$ by definition of the augmented base locus.\\
Now first suppose that $a_1\le 0$. Then as $F\nsubseteq B_+(\overline{D})$ by hypothesis, by adding $(-a_1)\overline{D}$ we deduce using our discussion above that $F\nsubseteq B_+(a_2\overline{E})=B_+(\overline{E})(\ \text{resp.}\ B_+(-\overline{E}))$ if $a_2> 0$( resp. $a_2<0$) by the remark at the end of Definition \ref{def:augloc} and clearly $a_2\neq 0$. Then switching $\overline{D}$ with $\overline{E}$( resp. $-\overline{E})$ we can assume that $a_1>0$. Then once again by the remark, we conclude that $F\nsubseteq B_+(a_1\overline{D}+a_2\overline{E})=B_+(\overline{D}+q\overline{E})$ for $q=\frac{a_2}{a_1}$. Once we obtain this then thanks to Lemma \ref{lemma:restrbiginterior} we can argue exactly as in the end of the proof of Lemma \ref{lemma:biginter} to obtain the claim.
\end{proof}
Next we show that the interior of the the support is actually non-empty- To show this we show that the graded semi-group generates the whole $\mathbb{Z}^{d+2}$ in our next Lemma.
\begin{lemma}
\label{lemma:restrnonemptint}
Suppose $\overline{D}$ and $\overline{E}$ are adelic divisors on a normal quasi-projective variety $U$ over $K$ and $F$ a closed irreducible sub-variety of $U$ with $F\nsubseteq B_+(\overline{D})$. Then $\Gamma_{U|F}(F)$ generates $\mathbb{Z}^{d+2}$ as a group.
\end{lemma}
\begin{proof}
The proof is almost identical to the proof of Lemma \ref{lemma:nonemptint}. We just need to note that by the proof of Lemma \ref{lemma:restrbiginterior}, when $q=0$ we can find a positive integer $n$ such that $F\nsubseteq B_+(\overline{D}-\frac{1}{n}\overline{E})=B_+(n\overline{D}-\overline{E})$. The rest of the argument is identical to Lemma \ref{lemma:nonemptint} thanks to the third property in Lemma \ref{lemma:restrfingen} and as $(1,0)$ and $(n,-1)$ generate $\mathbb{Z}^2$ as a group. 
\end{proof}
Finally we are ready to state and prove the main theorem of this section.
\begin{theorem}
\label{theorem:restrglobalbody}
Suppose $\overline{D}$ and $\overline{E}$ be adelic divisors on a normal quasi-projective variety $U$ over $K$ and let $F$ be an irreducible closed sub-variety such that $F\nsubseteq B_+(\overline{D})$. Then there exists a convex body $\Delta_{U|F}(F)=\Delta_{U|F}(F,\overline{D},\overline{E})\subset \mathbb{R}^{d+2}$ with the property that for any $\vec{a}=(a_1,a_2)\in\mathbb{Q}^2$ with $F\nsubseteq B_+(a_1\overline{D}+a_2\overline{E})$, we have \[\Delta_{U|F}(a_1\overline{D}+a_2\overline{E})=\Delta_{U|F}(F)\cap(\mathbb{R}^d\times\{\vec{a}\})\]
where $\Delta_{U|F}(a_1\overline{D}+a_2\overline{E})$ is the restricted Okounkov body of $a_1\overline{D}+a_2\overline{E}$ as constructed in section 2.
\end{theorem}
\begin{proof}
Clearly it is enough to show when $\vec{a}\in\mathbb{Z}^2$ by homogeneity of Okounkov bodies( Lemma \ref{lemma:restrhomogen}). Note that the semi-group $\Gamma_{U|F}(a_1\overline{D}+a_2\overline{E})$ sits naturally in $\mathbb{N}^d\times \mathbb{N}\cdot\vec{a}\cong \mathbb{N}^{d+1}$ and by construction of $\Delta_{U|F}(\cdot)$, we deduce that $\Delta_{U|F}(a_1\overline{D}+a_2\overline{E})=\Sigma(\Gamma_{U|F}(a_1\overline{D}+a_2\overline{E})_{\mathbb{N}\vec{a}})\cap(\mathbb{R}^d\times\{\vec{a}\})$. By Lemma \ref{lemma:restrbiginterior} we get that $\vec{a}\in\text{int}(\text{Supp}(\Delta(U))$ and hence by Lemma \ref{lemma:interior} we have $\Delta_{U|F}(F)_{\mathbb{R}\vec{a}}=\Sigma(\Gamma_{U|F}(a_1\overline{D}+a_2\overline{E})_{\mathbb{N}\vec{a}})$. Hence we deduce that 
\[\Delta_{U|F}(a_1\overline{D}+a_2\overline{E})=\Sigma(\Gamma_{U|F}(a_1\overline{D}+a_2\overline{E})_{\mathbb{N}\vec{a}})\cap(\mathbb{R}^d\times\{\vec{a}\})=\Delta_{U|F}(F)_{\mathbb{R}\vec{a}}\cap(\mathbb{R}^d\times\{\vec{a}\})=\Delta_{U|F}(F)\cap(\mathbb{R}^d\times\{\vec{a}\})\]
concluding the proof.
\end{proof}
\subsection{Corollaries}
In this section we deduce some corollaries which are direct from the existence of global bodies for restricted volumes as shown in Theorem \ref{theorem:restrglobalbody}. Note that we already have the notion of restricted volume of a line bundle $L$ along the closed sub-variety $E$ of a projective variety $X$  defined similarly as defined before Lemma 2.16 in \cite{lazarsfeld2008convex} which we denote by \emph{projective restricted volume} in the next corollary.
\begin{corollary}
\label{corol:restrcoonvbod}
Suppose $\overline{D}$ is an adelic divisor on a normal quasi-projective variety $U$ over $K$ and suppose $F$ is a closed irreducible sub-variety of $U$ with $F\nsubseteq B_+(\overline{D})$. Furthermore suppose $\overline{D}$ is given by a Cauchy sequence of models $\{X_i,D_i\}$ and let $F_j$ be the Zariski closure of $F$ in $X_j$. Then we have
\[\lim_{i\to\infty}d_H(\Delta_{U|F}(\overline{D}),\Delta_{U|F}(\overline{D_i}))=0\]
where $d_H(\cdot,\cdot)$ is the Hausdorff metric and $\overline{D_i}$ is $D_i$ considered as a model adelic divisor. In particular, we have
\[\widehat{\emph{vol}}_{U|F}(\overline{D})=\lim_{i\to\infty}\emph{vol}_{X_i|F_i}(O(D_i))\]
where $\emph{vol}_{X_i|F_i}(O(D_i))$ is the projective restricted volume of the line-bundle $O(D_i)$ with respect to $F_i$.
\end{corollary}
\begin{proof}
The proof is very similar to that of the proof of Lemma \ref{corol:coonvbod}. We begin by noting the set of inclusions 
\[\Delta_{U|F}(\overline{D}-q_j\overline{D_0})\subseteq \Delta_{U|F}(\overline{D_j})\subseteq \Delta_{U|F}(\overline{D}+q_j\overline{D_0})\]
where we put overlines to emphasize that they are looked as model divisors. Now the first claim follows once again noting that the two extremities of the above inclusions converge under the Hausdorff metric thanks to Theorem \ref{theorem:restrglobalbody} and Theorem 13 of \cite{Khovanskii2012-sb} when $q_j$ is small enough. Then note that from Lemma \ref{lemma:restrbiginterior}, as $F\nsubseteq B_+(\overline{D})$ we conclude that $F\nsubseteq B_+(\overline{D}-q_j\overline{D_0})\supseteq B_+(\overline{D_j})$ for large enough $j$ as $q_j\to 0$. Hence for large enough $j$ we have $F\nsubseteq B_+(\overline{D_j})$ which implies $\text{vol}(\Delta_{U|F}(\overline{D_j}))=\frac{1}{k!}\widehat{\text{vol}}_{U|F}(\overline{D_j})=\frac{1}{k!}\text{vol}_{X_j|F_j}(O(D_j))$ thanks to Theorem $\ref{theorem:restrvol}$ which now clearly gives the second claim together with the first claim.
\end{proof}
\begin{corollary}[log-concavity]
\label{corol:restrlogconcave}
Suppose $\overline{D_i}$ are two adelic divisors on a normal quasi-projective variety $U$ over $K$ for $i=1,2$. Furthermore suppose $F$ is a closed irreducible sub-variety of $U$ with $F\nsubseteq \emph{Bs}(\overline{D_i})$ for $i=1,2$. Then we have 
\[\widehat{\emph{vol}}_{U|F}(\overline{D_1}+\overline{D_2})^{\frac{1}{k}}\ge \widehat{\emph{vol}}_{U|F}(\overline{D_1})^{\frac{1}{k}}+\widehat{\emph{vol}}_{U|F}(\overline{D_2})^{\frac{1}{k}}\]
where $\emph{dim}(E)=k$.
\end{corollary}
\begin{proof}
 When $F\nsubseteq B_+(\overline{D_i})$ for both $i$, so is their sum and hence passing to models, we are reduced to the claim in the projective setting thanks to Corollary \ref{corol:restrcoonvbod}. The projective case can be deduced from the existence of global bodies as indicated in Example 4.22 of \cite{lazarsfeld2008convex}.
\end{proof}
\section*{Acknowledgements}
The author thanks Walter Gubler and Roberto Gualdi for numerous fruitful discussions in the process of preparation of this article.

\printbibliography

\end{document}